\documentclass[preprint]{imsart}

\usepackage{amsthm,amsmath}
\RequirePackage{natbib}
\RequirePackage[colorlinks,citecolor=blue,urlcolor=blue]{hyperref}
\usepackage{amssymb}
\usepackage{accents,xcolor}
\usepackage{mathtools}
\usepackage{mleftright}
\usepackage{lscape} 
\usepackage{xfrac,enumerate}

\startlocaldefs
\newcommand{\ex}{\mathbb{E}}

\newcommand{\pr}{\mathbb{P}}
\newcommand{\R}{\mathbb{R}}
\newcommand{\bigo}{\mathcal{O}}
\newcommand{\N}{\mathcal{N}}
\DeclarePairedDelimiterX{\norm}[1]{\lVert}{\rVert}{#1}
\DeclarePairedDelimiterX{\abs}[1]{\lvert}{\rvert}{#1}

\newtheorem{definition}{Definition}[section]
\newtheorem{lemma}[definition]{Lemma}
\newtheorem{theorem}[definition]{Theorem}

\newtheorem{Coro}[definition]{Corollary}
\theoremstyle{definition} 
\newtheorem{example}[definition]{Example}

\newtheorem*{remark}{Remark}

\endlocaldefs

\usepackage{Sweave}
\begin{document}
\Sconcordance{concordance:paperOITV.tex:paperOITV.Rnw:%
1 76 1 1 0 50 1}
\Sconcordance{concordance:paperOITV.tex:./Section1.Rnw:ofs 128:%
1 392 1}
\Sconcordance{concordance:paperOITV.tex:./Section2.Rnw:ofs 521:%
1 81 1}
\Sconcordance{concordance:paperOITV.tex:./Section3.Rnw:ofs 603:%
1 167 1}
\Sconcordance{concordance:paperOITV.tex:./Section4.Rnw:ofs 771:%
1 117 1}
\Sconcordance{concordance:paperOITV.tex:./Section5.Rnw:ofs 889:%
1 166 1}
\Sconcordance{concordance:paperOITV.tex:./Section6.Rnw:ofs 1056:%
1 55 1}
\Sconcordance{concordance:paperOITV.tex:./Section7.Rnw:ofs 1112:%
1 248 1}
\Sconcordance{concordance:paperOITV.tex:./Section8.Rnw:ofs 1361:%
1 6 1}
\Sconcordance{concordance:paperOITV.tex:paperOITV.Rnw:ofs 1368:%
136 7 1}
\Sconcordance{concordance:paperOITV.tex:./AppendixE.Rnw:ofs 1376:%
1 155 1}
\Sconcordance{concordance:paperOITV.tex:./AppendixB.Rnw:ofs 1532:%
1 325 1}
\Sconcordance{concordance:paperOITV.tex:./AppendixC.Rnw:ofs 1858:%
1 159 1}
\Sconcordance{concordance:paperOITV.tex:./AppendixA.Rnw:ofs 2018:%
1 673 1}
\Sconcordance{concordance:paperOITV.tex:paperOITV.Rnw:ofs 2692:%
148 12 1}

\begin{frontmatter}
\title{On the total variation regularized estimator over a class of tree graphs}
\runtitle{Total variation regularization over graphs}


\author{\fnms{Francesco} \snm{Ortelli}\corref{}\ead[label=e1]{francesco.ortelli@stat.math.ethz.ch}}\and \author{\fnms{Sara} \snm{van de Geer}\ead[label=e2]{geer@stat.math.ethz.ch}}
\address{R\"{a}mistrasse 101\\ 8092 Z\"{u}rich\\ \printead{e1}; \printead{e2}}

\affiliation{Seminar for Statistics, ETH Z\"{u}rich}

\runauthor{Ortelli, van de Geer}

\begin{abstract}
We generalize to tree graphs obtained by connecting path graphs an oracle result obtained for the Fused Lasso over the path graph. Moreover we show that it is possible to substitute in the oracle inequality the minimum of the distances between jumps by their harmonic mean. In doing so we prove a lower bound on the compatibility constant for the total variation penalty. Our analysis leverages insights obtained for the path graph with one branch to understand the case of more general tree graphs.
 As a side result, we get insights into the irrepresentable condition for such tree graphs.
\end{abstract}



\end{frontmatter}
\tableofcontents

\section{Introduction}\label{sec1}

The aim of this paper is to refine and extend to the more general case of ``large enough'' tree graphs the approach used by \cite{dala17} to prove an oracle inequality for the Fused Lasso estimator, also known as total variation regularized estimator. As a side result, we will obtain some insight into the irrepresentable condition for such ``large enough'' tree graphs.

The main reference of this article is \cite{dala17}, who consider the path graph.  We refine and generalize their approach (i.e. their Theorem 3, Proposition 2 and Proposition 3) to the case of more general tree graphs. The main refinements we prove are an oracle theory for the total variation regularized estimators over trees when the first coefficient is not penalized, a proof of an (in principle tight) lower bound for the compatibility constant and, as a consequence of this bound, the substitution in the oracle bound of the minimum of the distances between jumps by their harmonic mean. We elaborate the theory from the particular case of the path graph to the more general case of tree graphs which can be cut into path graphs. The tree graph with one branch is in this context the simplest instance of such more complex tree graphs, which allows us to develop insights into more general cases, while keeping the overview.

The paper is organized as follows: in Section \ref{sec1} we expose the framework together with a review of the literature on the topic; in Section \ref{sec2} we refine the proof of Theorem 3 of \cite{dala17} and adapt it to the case where one coefficient of the Lasso is left unpenalized: this proof will be a working tool for  establishing  oracle inequalities for total variation penalized estimators; in Section \ref{sec3} we expose how to easily compute objects related to projections which are needed for finding explicit bounds on weighted compatibility constants and when the irrepresentable condition is satisfied; in Section \ref{sec4} we present a tight lower bound for the (weighted) compatibility constant for the Fused Lasso and use it with the approach exposed in Section \ref{sec2} to prove an oracle inequality; in Section \ref{sec5} we generalize Section \ref{sec4} to the case of the branched path graph; Section \ref{sec6} presents further extensions to more general tree graphs; Section \ref{sec7} handles the asymptotic pattern recovery properties of the total variation regularized estimator on the (branched) path graph exposes an extension to more general tree graphs; Section \ref{sec8} concludes the paper. 

\subsection{General framework}

We study total variation regularized estimators on graphs, their oracle properties and their asymptotic pattern recovery properties.

For a vector $v\in\R^n$ we write $\norm{v}_1=\sum_{i=1}^n \abs{v_i}$ and $\norm{v}^2_n=\frac{1}{n}\sum_{i=1}^n v_i^2$.

Let $\mathcal{G}=(V,E)$ be a graph, where $V$ is the set of vertices and $E$ is the set of edges. Let $n:=\abs{V}$ be its number of vertices and $m:=\abs{E}$ its number of edges.
Let the elements of $E$ be denoted by $e(i,j)$, where $i,j\in V$ are the vertices connected by an edge.

Let $D_{\mathcal{G}}\in\R^{m\times n}$ denote the \textbf{incidence matrix} of a graph $\mathcal{G}$, defined as
\begin{equation*}
(D_e)_k=\begin{cases} -1, & \text{if } k=\min(i,j)\\
 +1, & \text{if } k=\max(i,j)\\
 0, & \text{else}, 
\end{cases}
\end{equation*}
where $D_e\in\R^n$ is the row of $D_{\mathcal{G}}$ corresponding to the edge of $e(i,j)$. 

Let $f\in\R^n$ be a function defined at each vertex of the graph.
The \textbf{total variation} of $f$ over the graph $\mathcal{G}$ is defined as
\begin{equation*}
\text{TV}_{\mathcal{G}}(f):= \norm{D_{\mathcal{G}}f}_1 =\sum_{e(i,j)\in E} \abs{f_j-f_i} .
\end{equation*}

Assume we observe the values of a signal $f^0\in\R^n$ contaminated with some Gaussian noise $\epsilon \sim \N_n(0,\sigma^2 \text{I}_n)$, i.e. $Y=f^0+\epsilon$.
The \textbf{total variation regularized estimator} $\widehat{f}$ of $f^0$ over the graph $\mathcal{G}$  is defined as 
\begin{equation*}
\widehat{f}:=\arg\min_{f\in\R^n}\left\{\norm{Y-f}^2_n+2\lambda\norm{D_{\mathcal{G}}f}_1 \right\},
\end{equation*}
where $\lambda>0$ is a tuning parameter. This is a special case of the generalized Lasso with design matrix $\text{I}_n$ and penalty matrix $D_\mathcal{G}$. Hereafter we suppress the subscript $\mathcal{G}$ in the notation of the incidence matrix of the graph $\mathcal{G}$.

In this article, we restrict our attention to tree graphs, i.e. connected graphs with $m=n-1$.  For a tree graph we have that $D\in\R^{(n-1)\times n}$ and $\text{rank}(D)=n-1$. In order to manipulate the above problem to obtain an (almost) ordinary Lasso problem, we define $\widetilde{D}$, the \textbf{incidence matrix rooted at vertex $i$},  as 

\begin{equation*}
\widetilde{D}:=\begin{bmatrix} A\\ D \end{bmatrix}\in\R^{n\times n},
\end{equation*}

where 

\begin{equation*}
A:=(0,\ldots, 0,\underbrace{1}_{i},0,\ldots,0)\in\R^n
\end{equation*}

In the following, we are going to root the incidence matrix at the vertex $i=1$, obtaining in this way a lower triangular matrix with ones on the diagonal, and minus ones as nonzero off-diagonal elements. The quadratic matrix $\widetilde{D}$ is invertible and we denote its inverse by $X:= {\widetilde{D}}^{-1}$.

We  now perform a change of variables. Let $\beta:=\widetilde{D}f$, then $f=X\beta$.
The above problem can be rewritten as
\begin{equation*}
\widehat{\beta}=\arg\min_{\beta\in\R^n}\left\{\norm{Y-X\beta}_n^2 +2\lambda\sum_{i=2}^n \abs{\beta_i} \right\},
\end{equation*}
i.e. an ordinary Lasso problem with $p=n$, where the first coefficient $\beta_1$ is not penalized. Note that, in order to perform this transformation, it is necessary that we restrict ourselves to tree graphs, since we want $\widetilde{D}$ to be invertible.

Let $X=(X_1,X_{-1})$, where $X_1\in\R^n$ denotes the first column of $X$ and $X_{-1}\in\R^{n\times (n-1)}$ the remaining $n-1$ columns of $X$. Let $\beta_{-1}\in \R^{n-1}$ be the vector $\beta$ with the first entry removed.  Thanks to some easy calculations and denoting by $\widetilde{Y}$ and $\widetilde{X}_{-1}$ the column centered versions of $Y$ and $X_{-1}$, it is possible to write
\begin{equation*}
\widehat{\beta}_{-1}=\arg\min_{\beta_{-1}\in\R^{n-1}}\left\{ \norm{\widetilde{Y}-\widetilde{X}_{-1}\beta_{-1}}^2_n+2\lambda\norm{\beta_{-1}}_1 \right\}
\end{equation*}
and
\begin{equation*}
\widehat{\beta}_1=\frac{1}{n} \sum_{i=1}^n Y_i-(X_{-1})_i\widehat{\beta}_{-1},
\end{equation*}
and both $\widehat{\beta}_{-1}$ and $\widehat{\beta}_1$ depend on $\lambda$.

Note that prediction properties of $\widehat{\beta}$, i.e. the properties of $X\widehat{\beta}$, will translate into properties of the estimator $\widehat{f}$, often also called Edge Lasso estimator.

\begin{remark}
In the construction of an invertible matrix starting from $D$, it would be possible to choose $A=(1, \ldots, 1)$ as well. Indeed, when we perform the change of variables from $f$ to $\beta$, $\hat{\beta}_{-1}$ estimates the jumps and thus gives information about the relative location of the signal. However to be able to estimate the absolute location of the signal we either need an estimate of the absolute location of the signal at one point (choice $A:=(0,\ldots, 0,1,0,\ldots,0)$, $\hat{\beta}_{i}=\hat{f}_{i}$, in particular we consider the case $i=1$), or of the ``mean'' location of the signal (choice $A=(1, \ldots, 1)$, $\hat{\beta}_1=\sum_{i=1}^n \hat{f}_i$).
\end{remark}

\subsection{The path graph and the path graph with one branch}

In this article we are interested, besides the more general case of ``large enough'' tree graphs, in the particular cases of $D$ being the incidence matrix of either the path graph or the path graph with one branch.
The choice of $A$ makes it easy to calculate the matrix $X$ and gives a nice interpretation of it.

Let $P_1$ be the \textbf{path matrix} of the graph $\mathcal{G}$ with reference root the vertex 1. The matrix $P_1$ is constructed as follows:
\begin{equation*}
(P_1)_{ij}:= \begin{cases} 1, & \text{ if the vertex $j$ is on the path from vertex 1 to vertex $i$,}\\
0, &\text{ else.}
\end{cases}
\end{equation*}

\begin{theorem}[Inversion of the rooted incidence matrix]\label{invroot}
For a tree graph, the rooted incidence matrix $\widetilde{D}$ is invertible and
\begin{equation*}
X=\widetilde{D}^{-1}=P_1.
\end{equation*}
\end{theorem}

\begin{proof}[Proof of Theorem \ref{invroot}]
For a formal proof we refer to \cite{jaco08} and to \cite{bapa14}.
The intuition behind this theorem is to proceed as follows. We have to check that $\text{rank}(\widetilde{D})=n$. One can perform Gaussian elimination on the rooted incidence matrix. Keep the first row as it is and for row $i$ add up the rows indexed by the vertices belonging to the path going from vertex 1 to vertex $i$. In this way one can obtain an identity matrix and thus $\text{rank}(\widetilde{D})=n$. Similarly one can find the inverse, which obviously corresponds to $P_1$. 
\end{proof}


\begin{example}[Incidence matrix and path matrix with reference vertex 1 for the path graph]
Let $\mathcal{G}$ be the path graph with $n=6$ vertices. The incidence matrix is
\begin{equation*}
D=\begin{pmatrix*}[r]
-1&1& & & &  \\
 &-1&1& & &   \\
 & &-1&1& &   \\
 & & &-1&1&   \\
 & & & &-1&1  \\
 
\end{pmatrix*}
\in\R^{5\times 6}
\end{equation*}
and the path matrix with reference vertex 1 is
\begin{equation*}
X=\begin{pmatrix*}[r]
1& & & & &  \\
1&1& & & &  \\
1&1&1& & &  \\
1&1&1&1& &  \\
1&1&1&1&1&  \\
1&1&1&1&1&1\\

\end{pmatrix*}\in\R^{6\times 6}.
\end{equation*}
\end{example}

\begin{example}[Incidence matrix and path matrix with reference vertex 1 for the path graph with a branch]
Let $\mathcal{G}$ be the path graph with one branch. The graph has in total $n=n_1+n_2$ vertices. The main branch consists in $n_1$ vertices, the side branch in $n_2$ vertices and is attached to the vertex number $b<n_1$ of the main branch.
Take $n_1=4$, $n_2=2$ and $b=2$. The incidence matrix is
\begin{equation*}
D=\begin{pmatrix*}[r]

 -1&1& & & &  \\
  &-1&1& & &  \\
  & &-1&1& &  \\
&-1 & & &1&  \\
 & & & &-1&1 \\
\end{pmatrix*}\in\R^{5\times 6}
\end{equation*}
and the path matrix with reference vertex 1 is
\begin{equation*}
X=\begin{pmatrix*}[r]

1& & &  & & \\
1&1& &  & & \\
1&1&1&  & & \\
1&1&1&1&  & \\
1&1& & & 1& \\
1&1& &  &1&1\\
\end{pmatrix*}\in\R^{6\times 6}.
\end{equation*}
\end{example}

\subsubsection{Notation}\label{notation}
Here we expose the notational conventions used for handling the (branched) path graph and later branching points with arbitrarily many ($K$) branches.

\begin{itemize}
\item
\textbf{(Branched) path graph}

We decide to enumerate the vertices of the (branched) path graph starting from the root $1$, continuing up to the end of the main branch $n_1$ and then continuing from the vertex $n_1+1$ of the side branch attached to vertex $b$ up to the last vertex of the side branch $n=n_1+n_2$.

We are going to use two different notations: the one is going to be used for finding explicit expressions for quantities related to the projection of a column of $X$ onto some subsets of the columns of $X$. The other is going to be used when calculating the compatibility constant and is based on the decomposition of the (branched) path graph into smaller path graphs. In both notations we let the set $S\subseteq \{2, \ldots , n\} $ be a candidate set of active edges.

\bigskip

\textbf{First notation}

We partition $S$ into three mutually disjoint sets $S_1,S_2,S_3$, where $S_1\subseteq\{2,\ldots, b\}$, $S_2\subseteq\{b+1,\ldots, n_1\}$, $S_3\subseteq\{n_1+1,\ldots, n\}$.
We write the sets $S_1,S_2,S_3$ as:
\begin{equation*}
S_1=: \left\{i_1,\ldots, i_{s_1} \right\},
S_2=: \left\{j_1,\ldots, j_{s_2} \right\},
S_3=: \left\{k_1,\ldots, k_{s_3} \right\}.
\end{equation*}

Note that $s_i:=\abs{S_i}, i\in \{1,2,3 \}$ and $s:=\abs{S}=s_1+s_2+s_3$.

Let us write $S=\{\xi_1,\ldots, \xi_{s_1+s_2+s_3} \}$.  Define
\begin{eqnarray*}
B & = & \{ \xi_1 -1, \xi_2 -\xi_1, \xi_3 -\xi_2, \ldots, b-\xi_{s_1}+1,\\
& &\xi_{s_1+1}-b-1, \ldots, \xi_{s_1+s_2}-\xi_{s_1+s_2-1},  n_1-\xi_{s_1+s_2}+1, \\
& &\xi_{s_1+s_2+1}-n_1-1, \xi_{s_1+s_2+2} - \xi_{s_1+s_2+1}, \ldots, n-\xi_{s_1+s_2+s_3} +1 \}\\
&=:& \{b_{1},b_2, b_3, \ldots, b_{s_1+1},b_{s_1+2} ,\ldots, b_{s_1+s_2+1},b_{s_1+s_2+2},\\
& & b_{s_1+s_2+3},b_{s_1+s_2+4}, \ldots, b_{s_1+s_2+s_3+3} \}.
\end{eqnarray*}

Define $b^*:=b_{s_1+1}+b_{s_1+2}+b_{s_1+s_2+3}$.

In the case where we consider the path graph we simply take $S=S_1$ (i.e. $n=n_1$)

\bigskip

\textbf{Second notation} (for bounding the compatibility constant).

What is meant with this second notation is that we decompose the branched path graph into three smaller path graphs. However the end of the first one does not necessarily coincide with the point $b$ and the begin of the other two does not necessarily coincide with the points $b+1$ and $n_1+1$ respectively.

Let us write 
\begin{equation*}
S_1=\{d^1_1+1, d^1_1+d^1_2+1, \ldots, d^1_1+d^1_2+\ldots+ d^1_{s_1}+1 \}=S\cap\{1, \ldots, b\},
\end{equation*}
and
\begin{equation*}
S_i=\{p_i+1, p_i+d_2^i+1,p_i+d^i_2+d^i_3+1, \ldots, p_i+d^i_2+d^i_3+\ldots+ d^i_{s_i}+1  \}, i=2,3,
\end{equation*}
where, using the first notation introduced, $p_2=j_1-1$, $p_3=k_1-1$,  $d^2_{s_2+1}=n_1-\xi_{s_1+s_2}+1$ and $d^3_{s_3+1}=n-\xi_{s_1+s_2+s_3}+1$. Note that $b^*=d^1_{s_1+1}+d^2_1+d_1^3=b_{s_1+1}+ b_{s_1+2}+b_{s_1+s_2+3}$. 
We require, $\forall i$, $d^i_1\geq 2$,$d^i_j\geq 4, \forall j\in \{2, \ldots, s_i \}$, and $d^i_{s_i+1}\geq 2$. Let $u^i_j\in\mathbb{N}$ satisfy $2 \le u^i_j\le d^i_j-2$ for $j\in \{2, \ldots, s_i \}$ and $i\in \{1,2,3\}$.
The elements $d_{s_1+1}^1,d^2_1,d^3_1$ are only constrained by the fact that they have to be greater or equal than two, otherwise, for a given $S$ their choice is left free.

Moreover note that $\sum_{i=1}^3 \sum_{j=1}^{s_1+1}d_j^i=n$. We thus end up with three sequences of integers $\{d_j^i\}_{i=1}^{s_i}, i\in\{1,2,3\}$.

\begin{remark}
We can relate part of these sequences to the set $B$ defined in the first notation.
Indeed,
\begin{itemize}
\item
$ \{d^1_{i} \}_{i=1}^{s_1}=\{b_{i} \}_{i=2}^{s_1}$;
\item
$ \{d^2_{i} \}_{i=2}^{s_2+1}=\{b_{i} \}_{i=s_1+3}^{s_1+s_2+2}$;
\item
$ \{d^3_{i} \}_{i=2}^{s_3+1}=\{b_{i} \}_{i=s_1+s_2+4}^{s_1+s_2+s_3+3}$.
\end{itemize}

We see that the only place where there might be some discrepancy between the first and the second notation is at $d^1_{s_1+1},d^2_1,d^3_1$, which might be different from $b_{s_1+1},b_{s_1+2},b_{s_1+s_2+3}$.
\end{remark}

In the case of the path graph we just consider one single of these path graphs and thus $S=S_1$ and $s=s_1$ and we omit the index $i$.

\item
\textbf{Branching point with arbitrarily many branches}

In Sections \ref{sec3} and \ref{sec6} we are going to consider branching points participating in $K+1$ edges. In these cases we are ging go denote by $b_1$ the number of vertices between the ramification point and the last vertex in $S$ in the main branch, with these two extreme vertices included, and by $b_2, \ldots, b_{K+1}$ the number of vertices after the ramification point and before the first vertex in $S$ (or the end of the relative branch). In these more complex cases for the sake of simplicity we only consider situations where the first and second notation coincide. We are often going to restrict our attention to \textbf{``large enough''} general tree graphs. These can be seen as tree graphs composed of $g$ path graphs glued together at their \textbf{extremities} with $d^i_j\ge 4, \forall j \in \{1, \ldots, s_i+1 \}, \forall i \in \{1, \ldots, g \}$. The reason of these requirements will become clear in Sections \ref{sec5} and \ref{sec6}.
\end{itemize}

\subsection{Review of the literature}

While to our knowledge there is no attempt in the literature to analyze the specific properties of the total variation regularized least squares estimator over general branched tree graphs, there is a lot of work in the field of the so called Fused Lasso estimator.
An early analysis of the Fused Lasso estimator
can be found in \cite{mamm97-2}. Some other early work is exposed in \cite{tibs05,frie07,tibs11}, where also computational aspects are considered.

In the literature we can find two main currents of research, the one focusing on the pattern recovery properties (which is going to be quickly exposed in Section \ref{sec7}) and the other on the analysis of the mean squared error to prove oracle inequalities.

\subsubsection{Minimax rates}

In this subsection we expose some results on minimax rates, making use of the notation found in \cite{sadh16}.
In particular, let
\begin{equation*}
\mathcal{T}(C)=\left\{f\in\R^n: \norm{Df}_1\leq C \right\}
\end{equation*}
be the class of (discrete) functions of bounded total variation on the path graph, where $D$ is its incidence matrix. Assume the linear model with $f^0\in\mathcal{T}(C)$ for some $C>0$ and with iid Gaussian noise with variance $\sigma\in(0,\infty)$  .
It has been shown in \cite{dono98} that the minimax risk over the class of functions with bounded total variation $\mathcal{R}(\mathcal{T}(C))$ satisfies
\begin{equation*}
\mathcal{R}(\mathcal{T}(C)):=\inf_{\widehat{f}} \sup_{f^0\in\mathcal{T}(C)}\ex[\norm{\widehat{f}-f^0}_n^2]\asymp (C/n)^{2/3}.
\end{equation*}

\cite{mamm97-2} prove that, if $\lambda\asymp n^{-2/3} C^{1/3}$, then the Fused Lasso estimator
achieves the minimax rate within the class $\mathcal{T}(C)$. 
\cite{sadh16} also point out, that estimators which are linear in the observations can not achieve the minimax rate within the class of functions of bounded total variation, since they are not able to adapt to the spatially inhomogeneous smoothness of some elements of this class.

%
%

\subsubsection{Oracle inequalities}

We expose some recent results, appeared in the papers by \cite{hutt16,dala17,lin17b,gunt17}. In particular we give the rates of the remainder term in the (sharp) oracle inequalities holding with high probability exposed in these papers.
\begin{itemize}
\item \textbf{\cite{hutt16}} obtain a quite general result, in the sense that it applies to any graph $\mathcal{G}$ with incidence matrix $D\in\R^{m\times n}$. In particular for the choice of the tuning parameter $\lambda=\sigma\rho\sqrt{2\log\left({em}/{\delta}\right)}/n, \delta\in(0,\frac{1}{2})$,  they obtain the rate 
$$\bigo\left(\frac{\abs{S}\rho^2}{n\kappa^2_D(S)}\log\left({em}/{\delta} \right)\right),$$
where, for a set $S\subseteq [m]$,
\begin{equation*}
\kappa_D(S):=\inf_{f\in\R^n}\frac{\sqrt{\abs{S}}\norm{f}_2}{\norm{(Df)_S}_1}, S\not= \emptyset
\end{equation*}
is called \textbf{compatibility factor} and $\rho$ is the largest $\ell^2$-norm of a column of the Moore-Penrose pseudoinverse $D^+=(\delta^+_1,\ldots,d^+_m)\in\R^{n\times m}$ of the incidence matrix $D$, i.e. $\rho=\max_{j\in[m]}\norm{\delta^+_j}_2$, and is called \textbf{inverse scaling factor}.

For the path graph, we have $m=n-1$, $\rho\asymp \sqrt{n}$ and, according to Lemma 3 in \cite{hutt16}, $\kappa_D(S)=\Omega\left(1 \right)$, if $\abs{S}\geq 2$.

\item \textbf{\cite{dala17}} obtain that, $\forall S\not=\emptyset$, for $\delta\in(0,\frac{1}{2})$ and the choice of the tuning parameter $\lambda:=2\sigma \sqrt{{2}\log\left({n}/{\delta} \right)/n}$, the remainder term has rate

$$\bigo\left(\frac{s\log n}{W_{\min,S}}+\frac{s \log^2 n}{n} \right),$$
where $S=\left\{i_1,\ldots, i_s \right\}$, $s=\abs{S}$, $W_{\min,S}:=\min_{2\leq j\leq s}\abs{i_j-i_{j-1}}$.

\item \textbf{\cite{lin17b}} prove a result similar to the one of \cite{dala17} using a technique that they call lower interpolant. Their result states that the mean squared error of the Fused Lasso estimator with the choice of the tuning patameter $\lambda=n^{-\frac{3}{4}} W_{\min,S_0}^{\frac{1}{4}}$ has error rate

$$
\bigo\left(\frac{s_0}{n}\left((\log s_0+\log\log n)\log n +\sqrt{\frac{n}{W_{\min,S_0}}}\right)\right).
$$

\item\textbf{\cite{gunt17}} consider the sequence of estimators $\{\widehat{f}_\lambda,\lambda\geq 0 \}$, where
\begin{equation*}
\widehat{f}_\lambda=\arg\min_{f\in\R^n}\left\{\norm{Y-f}^2_2+2\sigma\lambda\norm{Df}_1 \right\},
\end{equation*}
and prove that, when the minimum length condition $W_{\min,S_0}\geq \frac{c n}{s_0+1}, c\geq 1, $
is satisfied, then with high probability
\begin{equation*}
\inf_{\lambda\geq 0} \norm{\widehat{f}_\lambda-f^0}^2_n=\bigo\left(  \frac{s_0+1}{n}\log\left(\frac{ne}{s_0+1} \right)
\right).
\end{equation*}


\end{itemize}
\section{Approach for general tree graphs}\label{sec2}

The approach we follow is very similar to the one presented in the proof of Theorem 3 of \cite{dala17}. However, we refine their proof by not penalizing the first coefficient of $\beta$ and by adjusting the definition of compatibility constant accordingly. Note that by not penalizing the first coefficient we allow it to be always active. This is a more natural approach to utilize, considered our problem definition.

Let $\beta\in\R^n$ be a vector of coefficients, $S\subseteq\{2,\ldots, n \}$ a subset of the indices of $\beta$, called active set with $s:=\abs{S}$ being its cardinality.

\begin{definition}[\textbf{Compatibility constant}]\label{defcc}
The compatibility constant $\kappa(S)$ is defined as
\begin{equation*}
\kappa^2(S):=\min\left\{(s+1)\norm{X\beta}^2_n: \norm{\beta_S}_1-\norm{\beta_{-(\{1\}\cup S)}}_1=1 \right\}.
\end{equation*}
\end{definition}

Let  $V_{\{1\}\cup S}$ denote the linear subspace of $\R^n$ spanned by the columns of $X$ with index in $\{1\}\cup S$. Let $\Pi_{\{1\}\cup S}$ be the orthogonal projection matrix onto $V_{\{1\}\cup S}$. We have that $\Pi_{\{1\}\cup S}=X_{\{1\}\cup S}(X_{\{1\}\cup S}'X_{\{1\}\cup S})^{-1}X_{\{1\}\cup S}'$.

\begin{definition}
The vector $\omega\in\R^n$ is defined as
\begin{equation*}
\omega_j=\frac{\norm{X_j'(\text{I}-\Pi_{\{1\}\cup S})}_2}{\sqrt{n}},\forall j \in [n].
\end{equation*}
\end{definition}

\begin{remark}
Note that $\omega_{\{1\}\cup S}=0$ and $0 \leq \omega\leq 1$ since for tree graphs the maximum $\ell^2$-norm of a column of $X$ is $\sqrt{n}$.
\end{remark}

\begin{definition}
Take $\gamma>1$.
The vector of  weights $w\in\R^n$ is defined as
\begin{equation*}
w_j=1-\frac{\omega_j}{\gamma}, \forall j \in [n].
\end{equation*}
\end{definition}

\begin{remark}
Note that $0\leq w \leq 1$ and that $w_{\{1\}\cup S}=1$.
\end{remark}

For two vectors $a,b\in\R^k$, $a\odot b:=(a_1b_1,a_2b_2,\ldots,a_kb_k)'$. 

\begin{definition}[\textbf{Weighted compatibility constant}]\label{defwcc}
The weighted compatibility constant $\kappa_w(S)$ is defined as
\begin{equation*}
\kappa_w^2(S):=\min\left\{(s+1)\norm{X\beta}^2_n: \norm{(w\odot\beta)_S}_1-\norm{(w\odot \beta)_{-(\{1\}\cup S)}}_1=1 \right\}.
\end{equation*}
\end{definition}

\begin{remark}
Note that the (weighted) compatibility constant depends on the graph through $X$, which is the path matrix of the graph rooted at the vertex 1.
\end{remark}

\begin{remark}
Note that a key point in our approach is the computation of a lower bound for the compatibility constant over the path graph, which is shown to be tight in some special cases. The concept of compatibility constant for total variation estimators over graphs is already presented in \cite{hutt16}. However, we refer to the (different) definition given in \cite{dala17}, which we slightly modify to adapt it to our problem definition. 
\end{remark}

\begin{theorem}[Oracle inequality for total variation regularized estimators over tree graphs]\label{t21}
Fix $\delta\in (0,1)$ and $\gamma>1$. \\
Choose $ \lambda={\gamma\sigma}\sqrt{2\log\left({4(n-s-1)}/{\delta} \right)/n}$.
Then, with probability at least $1-\delta$, it holds that
\begin{eqnarray*}
\norm{\widehat{f}-f^0}_n^2&\leq& \inf_{f\in\R^n } \left\{\norm{f-f^0}_n^2 + 4\lambda\norm{(Df)_{-S}}_1 \right\}\\
&+&\frac{4\sigma^2}{n}\left((s+1)+2\log\left({2}/{\delta}\right)+\frac{\gamma^2 (s+1)}{\kappa^2_w(S)} \log \left({4(n-s-1)}/{\delta}  \right) \right).
\end{eqnarray*}
\end{theorem}

\begin{proof}[Proof of Theorem \ref{t21}]
See Appendix \ref{appE}
\end{proof}

\section{Calculation of projection coefficients and lengths of antiprojections, a local approach}\label{sec3}

In this section we are going to present an easy and intuitive way of calculating (anti-)projections and the related projection coefficients of the column of a path matrix rooted at vertex 1 of a tree onto a subset of the column of the same matrix. Let this matrix be called $X$. These calculations are motivated by the necessity of finding explicit expressions for the length of the antiprojections (for the weighted compatibility constant) and for the projection coefficients (to check for which signal patterns the irrepresentable condition is satisfied).

In particular consider the task of projecting a column $X_j, j\not \in \{1\}\cup S$ onto $X_{\{1\}\cup S}$. This can be seen as finding the following argmin:

$$
\hat{\theta}^j:=\arg\min_{\theta^j\in\R^{s+1}}\norm{X_j-X_{\{1\}\cup S}\theta^j}^2_2.
$$

We see that
\begin{itemize}
\item ${\hat{\theta}^j}'$ corresponds to the $j^{\text{th}}$ row of $X'X_{\{1\}\cup S}(X_{\{1\}\cup S}'X_{\{1\}\cup S})^{-1}$;
\item $\norm{X_j-X_{\{1\}\cup S}\hat{\theta}^j}^2_2=n\omega^2_j$.
\end{itemize}

The direct computation of these quantities can be quite laborious. Here, we show an easier way to compute these projections and we prove that they can be computed ``locally'', i.e. taking into account only some smaller part of the graph.

We start by considering the path graph. Then we treat the more general situation of ``large enough'' tree graphs.

\subsection{Path graph}

Let $j\not\in \{1\}\cup S$ be the index of a column of $X$ that we want to project onto $X_{\{1\}\cup S}$.
Define
\begin{equation}\label{jm}
j^- :=\max\left\{ i<j, i \in \{1\}\cup S \right\},
\end{equation}
\begin{equation}\label{jp}
j^+ :=\min\left\{ i>j, i \in \{1\}\cup S\cup \{ n+1\} \right\},
\end{equation}
and denote their indices inside $ \{1\}\cup S\cup \{ n+1\}=\{i_{1}, \ldots, i_{s+2} \}$ by $l^-$ and $l^+$, i.e. $j^-=i_{l^-}$ and $j^+=i_{l^+}$. We use the convention $X_{n+1}=0\in\R^n$.  We are going to show that the projection of $X_j$ onto $X_{\{1\}\cup S}$ is the same as its projection onto $X_{\{j^-\}\cup \{j^+ \}}$. This means that the part of the set $\{1\}\cup S$ not bordering with $j$ can be neglected.

The intuition behind this insight can be clarified as follows. Projecting $X_j$ onto $X_{\{1\}\cup S}$ amounts to finding the projection coefficients $\hat{theta}^j$ minimizing the length of the antiprojection. The projection is then $X_{\{1\}\cup S}\hat{\theta}^j$.
Since the columns of $X_{\{1\}\cup S}$ can be seen as indicator functions on $[n]$, this projection problem can be interpreted as the problem of finding the least squares approximation to $1_{\{i\ge j\}}$ by using functions in the class $\left\{1_{\{i\ge j^* \}}, j^*\in  \{1\}\cup S \right\}$.

We now apply a linear transformation in order to obtain orthogonal desing. Note that $\text{I}_{s+1}= \tilde{D}^{(s+1)}X^{(s+1)}$, where $\tilde{D}^{(s+1)}$ is the incidence matrix of a path graph with $s+1$ vertices rooted at vertex 1 and $X^{(s+1)}$ is its inverse, i.e. the corresponding rooted path matrix. We get that
\begin{equation*}
\min_{\theta^j\in\R^{s+1}}\norm{X_j-X_{\{1\}\cup S}\theta^j}^2_2=\min_{\tau^j\in\R^{s+1}}\norm{X_j-X_{\{1\}\cup S}\tilde{D}^{(s+1)} \tau^j}^2_2,
\end{equation*}
where $\tau^j=X^{(s+1)}\theta^j$, i.e. the progressively cumulative sum of the components of $\theta^j$ and $X_{\{1\}\cup S}\tilde{D}{(s+1)}\in\R^{n\times{(s+1)}}$ is a matrix containing as columns the indicator functions $\left\{1_{\{i_l\le i < i_{l+1} \}}, l\in\{1, \ldots, s+1 \} \right\}$, which are pairwise orthogonal.
Because of the orthogonality of the design matrix, we can now solve $s+1$ separate optimization problems to find the components of $\hat{\tau}^j$.
It is clear that, to  minimize the sum of squared residuals (i.e. the length of the antiprojection), $\hat{\tau}^j$ must be s.t.
$$ \{\hat{\tau}^j_i \}_{i<l^-}=0 \text{ and } \{\hat{\tau}^j_i \}_{i\ge l^+}=1.$$
It now remains to find $\hat{\tau}^j_{l^-}$ by solving
$$
\hat{\tau}^j_{l^-}=\arg\min_{x\in\R}\left\{(j-j^-)x^2+ (j^+-j)(1-x)^2 \right\}= \frac{j^+-j}{j^+-j^-}= 1-\frac{j-j^-}{j^+-j^-}.
$$
We see that, to get this projection coefficient, we either need to know $j^+$ and $j^-$ or the information on the length of the constant segment in which $j$ lies with its position within this segment. Thus we obtain that
$$
\hat{\tau}^j=\begin{pmatrix} 0 \\ \vdots \\ 0 \\  \frac{j^+-j}{j^+-j^-} \\ 1 \\ \vdots \\ 1 \end{pmatrix} \text{ and } \hat{\theta}^j=\begin{pmatrix} 0 \\ \vdots \\ 0 \\  \frac{j^+-j}{j^+-j^-} \\ \frac{j-j^-}{j^+-j^-}\\ 0 \\ \vdots \\ 0  \end{pmatrix},
$$

and have proved the following Lemma.

\begin{lemma}[Localizing the projections]\label{l61}
Let $X$ be the path matrix rooted at vertex 1 of a path graph with $n$ vertices and $S\subseteq \{ 2, \ldots, n \}$. For $j\not\in \{1\}\cup S$ define $j^-$ and $j^+$ as in Equations (\ref{jm}) and (\ref{jp}). Then
$$
\min_{\theta^j\in\R^{s+1}}\norm{X_j-X_{\{1\}\cup S}}^2_2= \min_{\tilde{\theta}^j\in\R^2}\norm{X_j-X_{\{j^-\}\cup \{ j^+\}}\tilde{\theta}^j}^2_2,
$$
i.e. the (length of the) (anti-)projections can be computed in a ``local'' way.

Moreover by writing $A_{\{1\}\cup S}=\text{I}_n-\Pi_{\{1\}\cup S}$ we have that
$$
\norm{A_{\{1\}\cup S}X_j}^2_2= \frac{(j^+-j)(j-j^-)}{(j^+-j^-)}.
$$

Furthermore, for $j<i_{s}, j\not \in \{1\}\cup S$, the sum of the entries of $\hat{\theta}^j$ is 1.
\end{lemma}

\subsection{General branching point}\label{s32}

Using arguments similar to the ones above we can now focus on a ramification point of a general tree graph. Let us consider $K$  path graphs attached at the end of one path graph (which we assume to contain the root). The path matrix rooted at the first vertex is
$$
X= \begin{pmatrix}
X^{(b_1)}  &    &    &   \\
1  &  X^{(b_2)}  &   &  \\
\vdots  &   &  \ddots  &  \\
1  &   &   &  X^{(b_{K+1})}
\end{pmatrix}
$$
and we want to find the projections of $X_{-1}$ onto $X_{1}=(1, \ldots, 1)'$. The entries $X^{(b_i)}, i \in \{1, \ldots, K+1\}$ of the matrix $X$ are $b_i\times b_i$ lower triangular matrices of ones. Let $b^*=\sum_{i=1}^{K+1}b_j$.
Let us write $j=1+i, i\in \{1, \ldots, b_1-1 \}$ and $j=\sum_{l=1}^{i^*}b_l-i, i\in \{1, \ldots, b_l \}, l\in \{2, \ldots, K+1 \}$. Without loss of generality we can consider only one $i^*\in \{2, \ldots, K+1 \}$.
We now consider two cases $l=1$ and $l\not=1$.

\begin{itemize}
\item First case: $l=1$.\\
We have
$$
\hat{\tau}^j_1=\arg\min_{x\in\R}\left\{ix^2+ (b^*-i)(1-x)^2 \right\}= 1-\frac{i}{b^*}
$$
and
$$
\norm{A_{\{1\}\cup S} X_{\sum_{i=1}^{i^*}b_i-j+1}}^2_2= \frac{i(b^*-i)}{b^*}, 1\le i \le b_1-1.
$$
\item Second case: $l\not= 1$.\\
We have
$$
\hat{\tau}^j_i=\arg\min_{x\in\R}\left\{i(1-x)^2+ (b^*-i)x^2 \right\}= \frac{i}{b^*}
$$
and
$$
\norm{A_{\{1\}\cup S} X_{j}}^2_2= \frac{i(b^*-i)}{b^*}, 1\le i \le b_l.
$$
\end{itemize}

Note that in the last region before the end of one branch, the approximation of the indicator function we implicitely calculate does not have to jump up to one and thus only one coefficient of the respective $\hat{\theta}^j$ will be nonzero and this coefficient will be smaller than one.

Now we focus on the case, where each of the branches (path graphs) involved in a ramification, presents at least one jump (i.e. one element of the set $S$). The length of the antiprojections is calculated in the same way as above. According to the arguments exposed in precedence, we can consider only the jumps surrounding the ramification point. Let us call them $j_1,j_2, \ldots, j_{k+1}$. We have to find
\begin{eqnarray*}
\hat{\theta}^j & = & \arg\min_{\theta^j\in \R^{s+1}}\norm{X_j-X_{\{1\}\cup S}\theta^j}^2_2\\
 & = &  \arg\min_{\tilde{\theta}^j\in\R^{K+1}} \norm{X_j-X_{\{j_1\}\cup \ldots \cup \{j_{K+1}\}}\tilde{\theta}^j}^2_2\\
 & = &  \arg\min_{\tilde{\theta}^j\in\R^{K+1}}\norm{X_j-X_{\{j_1\}\cup \ldots \cup \{j_{K+1}\}}D^{\star}X^{\star} \tilde{\theta}^j}^2_2,
\end{eqnarray*}
where
$$
D^{\star}=\begin{pmatrix}
1  &   &   &  \\
-1  &  1  &   &  \\
\vdots  &    &  \ddots  &  \\
-1  &    &   &  1 \\
\end{pmatrix}
\in\R^{(K+1)\times (K+1)}\text{ and }X^{\star}=\begin{pmatrix}
1  &   &   &  \\
1  &  1  &   &  \\
\vdots  &    &  \ddots  &  \\
1  &    &   &  1 \\
\end{pmatrix}
\in\R^{(K+1)}
$$
are respectively the rooted incidence matrix of a star graph with $(K+1)$ vertices and its inverse.

Let us write $j=j_1+i, i\in \{1, \ldots, b_1-1 \}$ and $j=j_l-i, i\in \{1, \ldots, b_l \}, l\in \{2, \ldots, K+1 \}$. Now let
$$
\hat{\tau}^j=\arg\min_{\tau^j\in \R^{K+1}}\norm{X_j-X_{\{j_1\}\cup \ldots \cup \{j_{K+1}\}}D^{\star}\tau^j}^2_2.
$$
We now consider two cases: $l=1$ and $l\not=1$.
\begin{itemize}
\item First case: $l=1$.\\
We have
\begin{align*}
\hat{\tau}^j_1 & =1-\frac{i}{b^*}\\
\hat{\tau}^j_l & =1,l=\{2, \ldots, K+1 \}, \\
\end{align*}
which translates into
\begin{align*}
\hat{\theta}^j_1 & =1-\frac{i}{b^*}\\
\hat{\theta}^j_l & =\frac{i}{b^*},l=\{2, \ldots, K+1 \}. \\
\end{align*}

\item Second case: $l\not=1$.\\
For $l=l'\not= 1$ we have
\begin{align*}
\hat{\tau}^j_1 & =\frac{i}{b^*}\\
\hat{\tau}^j_l & =0,l\in \{2, \ldots, K+1 \}\setminus \{l'\},\\
\hat{\tau}^j_l & =1,l=l', \\
\end{align*}
which translates into
\begin{align*}
\hat{\theta}^j_1  &  =\frac{i}{b^*}\\
\hat{\theta}^j_l & =-\frac{i}{b^*},l\in\{2, \ldots, K+1 \}\setminus \{l'\},\\
\hat{\theta}^j_l & =1-\frac{i}{b^*},l=l'. \\
\end{align*}
\end{itemize}

\section{Path graph}\label{sec4}

\subsection{Compatibility constant}

In this section we assume $\mathcal{G}$ to be the path graph with $n$ vertices.
We give two lower bounds for the compatibility constant for the path graph with and without weights. The proofs are postponed to the Appendix \ref{appB}, where we present some elements  that allow extension to the branched path graph and to more general tree graphs as well. These bounds are presented in a paper by \cite{vand18} as well. We use the second notation exposed in Subsection \ref{notation}.


\begin{lemma}[Lower bound on the compatibility constant for the path graph, part of Theorem 6.1 in \cite{vand18}]\label{l31}
For the path graph it holds that
\begin{equation*}
\kappa^2(S)\geq \frac{s+1}{ n}\frac{1}{K},
\end{equation*}
where
\begin{equation*}
K={1 \over d_1} + \sum_{j=2}^s \left({1 \over u_j}+ {1 \over d_j-u_j} \right) + {1 \over d_{s+1} }.
\end{equation*}
\end{lemma}

\begin{proof}[Proof of Lemma \ref{l31}]
See Appendix \ref{appB}.
\end{proof}

\begin{Coro}[The bound can be tight, part of Theorem 6.1 in \cite{vand18}]\label{c32}
Assume $d_j$ is even $\forall j \in \{2,\ldots, s \}$. Then we can take $u_j = d_j / 2$. Let us now define $f^{*}\in\R^n$ by
\begin{equation*}
f^{*}_i=\begin{cases} -{n \over d_1} & i=1 , \ldots , d_1 \cr 
   {2n \over d_2} & i=d_1+1 , \ldots , d_1 + d_2 \cr  
   \vdots & \ \cr (-1)^s {2n \over d_s} & i= \sum_{j=1}^{s-1} d_j +1 , \ldots ,\sum_{j=1}^s d_j  \cr 
   (-1)^{s+1} {n \over d_{s+1} } & i= \sum_{j=1}^s d_j +1 , \ldots , n \cr \end{cases} .
\end{equation*}
Let $\beta^{*}$ be defined by $f^{*}=X\beta^{*}$.
Then
\begin{equation*}
\kappa^2(S)= \frac{s+1}{n }\frac{1}{K} ,
\end{equation*}
where
\begin{equation*}
K={1 \over d_1} + \sum_{j=2}^{s} {4 \over d_j} + {1 \over d_{s+1} } .
\end{equation*}
\end{Coro}

\begin{proof}[Proof of Corollary \ref{c32}]
See Appendix \ref{appB}.
\end{proof}

\begin{remark}
For the compatibility constant we want to find the largest possible lower bound. Thus we have to choose the $u_j$'s s.t. $K$ is minimized. We look at the first order optimality conditions and notice that they reduce to finding the extremes of $(s-1)$ functions of the type $g(x)=\frac{1}{d-x}+\frac{1}{x}$, $x\in (0,d)$, where $t\in\mathbb{N}$ is fixed. The global minimum of $g$ on $(0,d)$ is achieved at $x=\frac{d}{2}$. Thus, we can not obtain the optimal value of $K$ as soon as at least one $d_j$ is odd.
\end{remark}

\begin{lemma}[Lower bound on the weighted compatibility constant for the path graph, Lemma 9.1 in \cite{vand18}]\label{l33}
For the path graph it holds that
\begin{equation*}
\kappa_w^2(S)\geq \frac{s+1}{n }\frac{1}{(\norm{w}_{\infty}\sqrt{K}+\norm{Dw}_2 )^2} \geq \frac{s+1}{n}\frac{1}{2 (\norm{w}^2_{\infty}K +\norm{Dw}^2_2)},
\end{equation*}
where $D$ is the incidence matrix of the path graph.
\end{lemma}

\begin{proof}[Proof of Lemma \ref{l33}]
See Appendix \ref{appB}.
\end{proof}

\subsection{Oracle inequality}

Define the vector
\begin{equation*}
\Delta:=\left( {d_1}, \lfloor{d_2}/{2}\rfloor,\lceil{d_2}/{2}\rceil, \ldots,\lfloor{d_s}/{2}\rfloor,\lceil{d_s}/{2}\rceil, {d_{s+1}} \right)\in\R^{s+1}
\end{equation*}
and let $\overline{\Delta}_h$ be its harmonic mean.

We now want to translate the result of Theorem \ref{t21} to the path graph. To do so we need
 a lower bound for the weighted compatibility constant, i.e. an explicit upper bound for $\sum_{i=2}^n(w_i-w_{i-1})^2$. In this way we can obtain the following Corollary.

\begin{Coro}[Sharp oracle inequality for the path graph]\label{c31}
Assume $d_i\ge 4, \forall i \in \{1, \ldots, s+1 \}$.
It holds that
\begin{eqnarray*}
\norm{\widehat{f}-f^0}^2_n &\leq& \inf_{f\in\R^n} \left\{\norm{f-f^0}^2_n+4\lambda\norm{(Df)_{-S}}_1 \right\}\\
& + &  {\frac{8\log(2/\delta)\sigma^2}{n} } + {4\sigma^2}\frac{s+1}{n}\\
& + &  8\sigma^2 \log(4(n-s-1)/\delta)\left( \frac{2\gamma^2 s}{ \bar{\Delta}_h}+5  \frac{s+1}{n}\log\left(\frac{n}{s+1} \right) \right).
\end{eqnarray*}
If we choose $f=f^0$ and $S=S_0$ we obtain that
\begin{eqnarray*}
\norm{\widehat{f}-f^0}^2_n &=& \bigo (\log(n)s_0   / \bar{\Delta}_h)+ \bigo (\log(n)\log( n/s_0)s_0/n) .
\end{eqnarray*}
\end{Coro}

\begin{proof}[Proof of Corollary \ref{c31}]
See Appendix \ref{appB}.
\end{proof}


\begin{remark}
Since the harmonic mean of $\Delta$  is upper bounded by its arithmetic mean, and this upper bound is attained when all the entries of $\Delta$ are the same, we get a lower bound for the order of the mean squared error of
$$
\frac{s\log(n)}{n}\left(s+\log\left(\frac{n}{s} \right) \right).
$$
\end{remark}

\begin{remark}
Our result differs from the one obtained by \cite{dala17} in two points:
\begin{itemize}
\item We have $\bar{\Delta}_h$, the harmonic mean of the distances between jumps, instead of $\min_j\Delta_j$, the minimum distance between jumps;
\item We slightly improve the rate from by reducing a $\log(n)$ to $\log(n/s)$. This is achieved with a more careful bound on the square of the consecutive differences of the weights.
\end{itemize}
\end{remark}

\section{Path graph with one branch}\label{sec5}

In this section we consider $\mathcal{G}$ to be the path graph with one branch and $n$ vertices.
%

\subsection{Compatibility constant}

\begin{lemma}[Lower bound for the compatibility constant for the branched path graph]\label{l41}
For the branched path graph it holds that
\begin{equation*}
\kappa^2(S)\geq \frac{s+1}{n}\frac{1}{K^b},
\end{equation*}
where
\begin{equation*}
K^b=\sum_{i=1}^3\left(\frac{1}{d_1^i}+ \sum_{j=2}^{s_i}\left({1\over u^i_j} +{1\over d^i_j-u^i_j }\right) +\frac{1}{d^i_{s_i+1}} \right)
\end{equation*}
\end{lemma}

\begin{proof}[Proof of Lemma \ref{l41}]
See Appendix \ref{appC}.
\end{proof}

\begin{Coro}[The bound can be tight]\label{l43}
Assume $d^i_j$ is even $\forall  j\in \{2, \ldots, s_i \},i\in \{1,2,3 \}$. One can then choose $u^i_j=d_j/2,\forall  j\in \{2, \ldots, s_i \},i\in \{1,2,3 \}$. Moreover, assume that $d^1_{s_1+1}=d^2_1=d^3_1$.  Let $f^i, i\in\{1,2,3 \}$ be the restriction of $f$ to the three path graphs of length $p_i$ each. Let us now define ${f^*}^i\in \R^{p_i}$ by 
\begin{equation*}
{f^*}^i_j=\begin{cases}
-\frac{n}{d^1_1} & j=1, \ldots, d^1_1\\
\frac{2n}{d^1_2} & j=d^1_1+1, \ldots, d^1_1+d^1_2\\
\vdots & \\
(-1)^{s_1}\frac{2n}{d^1_{s_1}} & j=\sum_{j=1}^{s_1-1}d^1_j+1, \ldots, \sum_{j=1}^{s_1}d^1_j\\
(-1)^{s_1+1}\frac{n}{d^1_{s_1+1}} & j=\sum_{j=1}^{s_1}d^1_j+1, \ldots, p_1\\
\end{cases}
\end{equation*}

and for $i\in \{2,3\}$

\begin{equation*}
{f^*}^i_j=\begin{cases}
(-1)^{s_1+1} \frac{n}{d^i_1} & j=1, \ldots, d^i_1\\
(-1)^{s_1+2}\frac{2n}{d^i_2} & j=d^i_1+1, \ldots, d^i_1+d^i_2\\
\vdots & \\
(-1)^{s_1+s_i+1}\frac{2n}{d^i_{s_i}} & j=\sum_{j=1}^{s_i-1}d^i_j+1, \ldots, \sum_{j=1}^{s_i}d^i_j\\
(-1)^{s_1+s_1+1}\frac{n}{d^i_{s_i+1}} & j=\sum_{j=1}^{s_i}d^i_j+1, \ldots, p_i.\\
\end{cases}
\end{equation*}
Let $\beta^*$ be defined by $f^*=X\beta^*$. Then
\begin{equation*}
\kappa^2(S)=\frac{s+1}{n }\frac{1}{K^b},
\end{equation*}
where
\begin{equation*}
K^b=\sum_{i=1}^{3}\left(\frac{1}{d^i_1}+\sum_{j=2}^{s_i} \frac{4}{d^i_{j}} + \frac{1}{d^i_{s_i+1}} \right).
\end{equation*}
\end{Coro}

\begin{proof}[Proof of Corollary \ref{l43}]
See Appendix \ref{appC}.
\end{proof}

Consider the decomposition of the branched path graph into three path graphs, implicitely done by using the \textbf{second notation} in Section \ref{notation}. Let $D^*$ denote the incidence matrix of the branched path graph, where the entries in the rows corresponding to the edges connecting the three above mentioned path graphs have been substituted with zeroes.

\begin{lemma}[Lower bound on the weighted compatibility constant for the branched path graph]\label{l42}

\begin{eqnarray*}
\kappa^2_w(S)&\geq& \frac{s+1}{n}\frac{1}{(\sqrt{K^b}\norm{w}_{\infty}+ \norm{D^*w}_2)^2} \geq\frac{s+1}{n}\frac{1}{2({K^b}\norm{w}^2_{\infty}+ \norm{D^*w}^2_2)}\\
&\geq& \frac{s+1}{n}\frac{1}{2({K^b}\norm{w}^2_{\infty}+ \norm{Dw}^2_2)} .
\end{eqnarray*}

\end{lemma}

\begin{proof}[Proof of Lemma \ref{l42}]
See Appendix \ref{appC}.
\end{proof}

\subsection{Oracle inequality}

As in the case of the path graph, to prove an oracle inequality for the branched path graph, we need to find an explicit expression to control the weighted compatibility constant to insert in Theorem \ref{t21}. The resulting bound is similar to the one obtained in the Proof of Corollary \ref{c31}, up to a difference: we now have to handle with care the region around the branching point $b$.

For the branched path graph we define the vectors
\begin{equation*}
\Delta^i := (d^i_1,\lfloor d^i_2/2 \rfloor ,\lceil d^i_2/2 \rceil, \ldots,\lfloor d^i_{s_i}/2 \rfloor,\lceil d^i_{s_i}/2 \rceil , d^i_{s_i+1})\in \R^{2s_i},
\end{equation*}
and $\Delta:=(\Delta^1,\Delta^2,\Delta^3)\in \R^{2s}$. Let $\bar{\Delta}_h$ be the harmonic mean of $\Delta$.

\begin{remark}
As made clear in the \textbf{second notation} in Section \ref{notation}, we require that all  $d^1_{s_1+1},d^2_1,d^3_1\geq 2$, i.e. $b^*=b_{s_1+1}+b_{s_1+2}+b_{s_1+s_2+3}\geq 6$. This means that our approach can handle the case where at most one of the jumps surrounding the bifurcation point occurs directly at the bifurcation point. Note that neither $b_{s_1+1}=0$ nor $b_{s_1+2}=b_{s_1+s_2+3}=0$ are allowed.
\end{remark}

We can distinguish the following four cases:
\begin{enumerate}[1)]
\item
$b_{s_1+1},b_{s_1+2},b_{s_1+s_2+3}\ge 2$;
\item
$b_{s_1+2}=0$ or $b_{s_1+s_2+3}=0$;
\item
$b_{s_1+1}=1$;
\begin{enumerate}[a)]
\item
$b_{s_1+2}\wedge b_{s_1+s_2+3}=2$;
\item
$b_{s_1+2}\wedge b_{s_1+s_2+3}\geq 3$;
\end{enumerate}

\item
$b_{s_1+2}=1$ or $b_{s_1+s_2+3}=1$;
\end{enumerate}

\begin{Coro}[Sharp oracle inequality for the branched path graph]\label{c41}
Assume that $d^1_1, d^2_{s_2+1}, d^3_{s_3+1}\ge 4$. It holds that
\begin{eqnarray*}
\norm{\widehat{f}-f^0}^2_n &\leq& \inf_{f\in\R^n} \left\{\norm{f-f^0}^2_n+4\lambda\norm{(Df)_{-S}}_1 \right\}\\
&+  &  {\frac{8\log(2/\delta)\sigma^2}{n} } + {4\sigma^2}\frac{s+1}{n}\\
&+& 8\sigma^2 \log(4(n-s-1)/\delta)\left(\frac{2\gamma^2 s}{\bar{\Delta}_h}+\frac{5(2s+3)}{2n}\log \left(\frac{n+1}{2s+3} \right)+ \frac{\zeta}{n} \right),
\end{eqnarray*}
where
\begin{equation*}
\zeta=\begin{cases}
0 &, \text{ Case 1)}\\
b^*/2 &,\text{ Case 2)}\\
3 &,\text{ Case 3)a)} \\
b^*/4 &,\text{ Case 3)b)}\\
b^*/4 &,\text{ Case 4)}
\end{cases}.
\end{equation*}
If we choose $f=f^0$ and $S=S_0$ we get that
\begin{equation*}
\norm{\hat{f}-f^0}^2_n=\bigo(\log (n)s_0 / \bar{\Delta}_h)+\bigo(\log (n)\log(n/s_0)s_0/n)+ \bigo(\log (n)\zeta /n).
\end{equation*}
\end{Coro}

\begin{proof}[Proof of Lemma \ref{c41}]
See Appendix \ref{appC}.
\end{proof}

%


\section{Extension to more general tree graphs}\label{sec6}

In this section we consider only situations corresponding to Case 1) of Corollary \ref{c41}. This means that we assume that, even when at the ramification point is attached more than one branch, the edge connecting the branch to the ramification point and the consecutive one do not present jumps (i.e. are not elements of the set $S$).

\subsection{Oracle inequality for general tree graphs}

With the insights gained in Section \ref{sec3} we can, by availing ourselves of simple means, prove an oracle inequality for a general tree graph, where the jumps in $S$ are far enough from the branching points, in analogy to Case 1) in Corollary \ref{c41}.

Here as well, we utilize the general approach exposed in Theorem \ref{t21} and we need to handle with care the weighted compatibility constant and find a lower bound for it.

We know that, when we are in (the generalization of) Case 1) of Corollary \ref{c41}, to prove bounds for the compatibility constant, the tree graph can be seen as a collection of path graphs glued together at (some of) their \textbf{extremities}. As seen in Section \ref{sec3}, the length of the antiprojections for the vertices around ramification points depends on all the branches attached to the ramification point in question. Here, for the sake of simplicity, we assume that $d^i_j\ge 4,\forall j, \forall i$, i.e. between consecutive jumps there are at least four vertices as well as there are at least four vertices before the first and after the last jump of each path graph resulting from the decomposition of the tree graph. This is what we call a ``large enough'' tree graph. Indeed, for $d^i_j\ge 4$, we have that $\log(d^i_j)\le 2\log(d^i_j/2)$.

Let $\mathcal{G}$ be a tree graph with the properties exposed above. In particular it can be decomposed into $g$ path graphs. For each of these path graphs, by using the second notation in Subsection \ref{notation}, we define the vectors
$$\Delta^i=( d^i_1, \lceil d^i_2/2 \rceil,\lfloor d^i_2/2 \rfloor, \ldots, \lceil d^i_{s_i}/2 \rceil,\lfloor d^i_{s_i}/2 \rfloor, d^i_{s_i+1} )\in \R^{2s_i}, i \in \{1, \ldots, g \}
$$
and
$$
\abs{\Delta^i}=( \lceil d^i_1/2 \rceil,\lfloor d^i_1/2 \rfloor, \ldots, \lceil d^i_{s_i+1}/2 \rceil,\lfloor d^i_{s_i+1}/2 \rfloor )\in \R^{2s_i+2}, i \in \{1, \ldots, g \}.
$$
Moreover we write
$$ \Delta= (\Delta^1, \ldots, \Delta^g)\in \R^{2s} \text{ and } \abs{\Delta}= (\abs{\Delta}^1, \ldots, \abs{\Delta}^g)\in \R^{2(s+g)}.
$$
We have that for $\mathcal{G}$,
$$
\kappa^s(S)\ge \frac{s+1}{n}\frac{1}{K}, K\le \frac{2s}{\bar{\Delta}_h},
$$
where $\bar{\Delta}_h$ is the harmonic mean of $\Delta$.
Moreover an upper bound for the inverse of the weighted compatibility constant can be computed by upper bounding the squared consecutive pairwise differences of the weigths for the $g$ path graphs.
We thus get that, in analogy to Corollary \ref{c31}
$$\frac{1}{\kappa^2_w(S)}\le \frac{2n}{s+1}\left(\frac{2s}{\bar{\Delta}_h}+ \frac{5}{\gamma^2}\frac{s+g}{n}\log \left(\frac{n}{s+g} \right) \right).
$$

We therefore get the following Corollary
\begin{Coro}[Oracle inequality for a general tree graph]
Let $\mathcal{G}$ be a tree graph, which can be decomposed in $g$ path graphs. Assume that $d^i_j\ge 4, \forall j \in \{ 1, \ldots, s_i+1\}, \forall i \in \{1, \ldots, g \}$. Then
\begin{eqnarray*}
\norm{\widehat{f}-f^0}^2_n  & \leq &  \inf_{f\in\R^n} \left\{\norm{f-f^0}^2_n+4\lambda\norm{(Df)_{-S}}_1 \right\}\\
 &  +  &   {\frac{8\log(2/\delta)\sigma^2}{n} } + {4\sigma^2}\frac{s+1}{n}\\
 &  +  &   8\sigma^2 \log(4(n-s-1)/\delta)\left( \frac{2\gamma^2s}{ \bar{\Delta}_h}+  5\frac{(s+g)}{n}\log\left(\frac{n}{s+g} \right) \right).
\end{eqnarray*}
\end{Coro}

\begin{remark}
Notice that it is advantageous to choose a decomposition where the path graphs are as large as possible, s.t. $g$ is small and less requirement on the $d^i_j$'s are posed.
\end{remark}

\begin{remark}
This approach is of course not optimal, however it allows us to prove in a simple way a theoretical guarantee for the Edge Lasso estimator if some (not extremely restrictive) requirement on $\mathcal{G}$ and $S$ is satisfied.
\end{remark}

\section{Asymptotic signal pattern recovery: the irrepresentable condition}\label{sec7}

\subsection{Review of the literature on pattern recovery}\label{prec}

Let $Y=X\beta^0+\epsilon, \epsilon\sim\N_n(0,\sigma^2\text{I}_n)$, where
$Y\in \R^n, X\in\R^{n\times p},\beta^0\in \R^p,\epsilon\in\R^n$. Let $S_0:=\left\{j\in[p]:\beta^0_j\not=0 \right\}$ be the active set of $\beta^0$ and $-S_0$ its complement. We are interested in the asymptotic sign recovery properties of the Lasso estimator
\begin{equation*}
\hat{\beta}:=\arg\min_{\beta\in\R^p}\left\{ \norm{Y-X\beta}^2_n+2\lambda\norm{\beta}_1\right\}.
\end{equation*}

\begin{definition}[\textbf{Sign recovery}, Definition 1 in \cite{zhao06}]
We say that an estimator $\hat{\beta}$ recovers the signs of the true coefficients $\beta^0$ if
\begin{equation*}
\text{sgn}(\hat{\beta})=\text{sgn}(\beta^0).
\end{equation*}
We then write
\begin{equation*}
\hat{\beta}=_s\beta^0.
\end{equation*}
\end{definition}

\begin{definition}[\textbf{Pattern recovery}]
We say that an estimator $\hat{f}$ of a signal $f^0$ on a graph $\mathcal{G}$ with incidence matrix $D$ recovers the signal pattern if
\begin{equation*}
D\hat{f}=_s Df^0.
\end{equation*}
\end{definition}

\begin{definition}[\textbf{Strong sign consistency}, Definition 2 in \cite{zhao06}]
We say that the Lasso estimator $\hat{\beta}$ is strongly sign consistent if $\exists \lambda=\lambda(n):$
\begin{equation*}
\lim_{n\to \infty}\pr\left(\hat{\beta}(\lambda)=_s\beta^0 \right)=1
\end{equation*}
\end{definition}

\begin{definition}[\textbf{Strong irrepresentable condition}, \cite{zhao06}]
Without loss of generality we can write
\begin{equation*}
\beta^0=\begin{pmatrix} \beta^0_{S_0}\\ \beta^0_{-S_0} \end{pmatrix}=\begin{pmatrix} \beta^0_{S_0}\\ 0 \end{pmatrix}=: \begin{pmatrix} \beta^0_{1}\\ \beta^0_{2} \end{pmatrix},
\end{equation*}
where 1 and 2 are shorthand notations for $S_0$ and $-S_0$ and
\begin{equation*}
\hat{\Sigma}:=\frac{X'X}{n}=\begin{pmatrix} \hat{\Sigma}_{11} & \hat{\Sigma}_{12}\\
\hat{\Sigma}_{21} & \hat{\Sigma}_{22} \end{pmatrix}.
\end{equation*}
Assume $\hat{\Sigma}_{11}$ and $\hat{\Sigma}_{22}$ are invertible. The strong irrepresentable condition is satisfied if $\exists \eta\in (0,1]:$
\begin{equation*}
\norm{\hat{\Sigma}_{21}\hat{\Sigma}_{11}^{-1}\text{sgn}(\beta^0_1)}_{\infty}\leq 1-\eta
\end{equation*}
\end{definition}

\cite{zhao06} prove (in Their Theorem 4) that under Gaussian noise the strong irrepresentable condition implies strong sign consistency of the Lasso estimator, if $\exists 0\le c_1<c_2\le 1$ and $C_1>0: s_0=\bigo(n^{c_1})$ and $n^{\frac{1-c_2}{2}}\min_{j\in S_0} \abs{\beta^0_j}\ge C_1$. For our setup this means that $s_0$ has to grow more slowly than $\bigo(n)$ and that the magnitude of the smallest nonzero coefficient has to decay (much) slower than $\bigo(n^{-1/2})$.

In the literature, considerable attention has been given to the question whether or not it is possible to consistently recover the  pattern of a piecewise constant signal contaminated with some noise, say Gaussian noise.
In that regard, \cite{qian16}  highlight the so called \textbf{staircase problem}: as soon as there are two consecutive jumps in the same direction in the underlying signal separated by a constant segment, no consistent pattern recovery is possible, since the irrepresentable condition (cfr. \cite{zhao06}) is violated. 

Some cures have been proposed to mitigate the staircase problem. \cite{roja15,otte16} suggest to modify the algorithm for computing the Fused Lasso estimator.
Their strategy is based on the connection made by \cite{roja14} between the Fused Lasso estimator and a sequence of discrete Brownian Bridges.
\cite{owra17} propose instead to normalize the design matrix of the associated Lasso problem, to comply with the irrepresentable condition. Another proposal aimed at complying with the irrepresentable condition is the one by \cite{qian16}, based on the preconditioning of the design matrix with the puffer transformation defined in \cite{jia15}, which results in estimating the jumps of the true signal with the soft-thresholded differences of consecutive observations.

\subsection{Approach to pattern recovery for total variation regularized estimators over tree graphs}

Let us now consider the case of the Edge Lasso on a tree graph rooted at vertex 1.
We saw in Section \ref{sec1} that the problem can be transformed into an ordinary Lasso problem where the first coefficient is not penalized.

We start with the following remark.

\begin{remark}[The irrepresentable condition when some coefficients are not penalized]
Let us consider the Lasso problem where some coefficients are not penalized, i.e. the estimator

\begin{equation*}
\hat{\beta}:=\arg\min_{\beta\in\R^p}\left\{\norm{Y-X\beta}^2_n+2\lambda\norm{\beta_{-U}}_1  \right\},
\end{equation*}
where $U,R,S$ are three subsets partitioning $p$. In particular $U$ is the set of the unpenalized coefficients, $R$ is the set of truly zero coefficients and $S$ is the set of truly nonzero (active) coefficients. We assume the linear model $Y=X\beta^0+\epsilon, \epsilon\sim\N_n(0,\sigma^2\text{I}_n)$. The vector of true coefficients $\beta^0$ can be written as
\begin{equation*}
\beta^0=\begin{pmatrix} \beta^0_{U} \\ \beta^0_{S}\\ 0 \end{pmatrix}.
\end{equation*}
Moreover we write
\begin{equation*}
\frac{X'X}{n}=:\hat{\Sigma}=
\begin{pmatrix}
\hat{\Sigma}_{UU} & \hat{\Sigma}_{US} & \hat{\Sigma}_{UR}\\
\hat{\Sigma}_{SU} & \hat{\Sigma}_{SS} & \hat{\Sigma}_{SR}\\
\hat{\Sigma}_{RU} & \hat{\Sigma}_{RS} & \hat{\Sigma}_{RR}
\end{pmatrix}
.
\end{equation*}
Assume that $\abs{U}\le n$ and that $\hat{\Sigma}_{UU}, \hat{\Sigma}_{SS}$ and $\hat{\Sigma}_{RR}$ are invertible. We can write the irrepresentable condition as
\begin{equation*}
\norm{X_R' A_U X_S(X_S'A_U X_S)^{-1}z^0_S}_{\infty}\le 1-\eta,
\end{equation*}
where $z^0_S=\text{sgn}(\beta^0_S)$, $A_U=\text{I}_n- \Pi_U$ is the antiprojection matrix onto $V_U$, the linear subspace spanned by $X_U$ and $\Pi_{U}:= X_U(X_U'X_U)^{-1}X_U'$ is the orthogonal projection matrix onto $V_U$.

Indeed, write $\delta:=\hat{\beta}-\beta^0$.
The KKT conditions can be written as
\begin{equation}\label{ekkt1}
\hat{\Sigma}_{UU}\delta_U+\hat{\Sigma}_{US}\delta_S+\hat{\Sigma}_{UR}\delta_R-\frac{X_U'\epsilon}{n}=0;
\end{equation}
\begin{equation}\label{ekkt2}
\hat{\Sigma}_{SU}\delta_U+\hat{\Sigma}_{SS}\delta_S + \hat{\Sigma}_{SR}\delta_R-\frac{X_S'\epsilon}{n}+\lambda \hat{z}_S=0, \hat{z}_S\in\delta\norm{\hat{\beta}_S}_1;
\end{equation}
\begin{equation}\label{ekkt3}
\hat{\Sigma}_{RU}\delta_U+\hat{\Sigma}_{RS}\delta_S + \hat{\Sigma}_{RR}\delta_R-\frac{X_R'\epsilon}{n}+\lambda \hat{z}_R=0, \hat{z}_R\in\delta\norm{\hat{\beta}_R}_1.
\end{equation}

By solving Equation \ref{ekkt1} with respect to $\delta_U$, then inserting into Equation \ref{ekkt2} and solving with respect to $\delta_S$, then inserting the expression for $\delta_R$ in the expression for $\delta_U$ to get $\delta_U(\delta_R)$ and $\delta_S(\delta_R)$ and by finally inserting them into Equation \ref{ekkt3}  by analogy with the proof proposed by \cite{zhao06}, we find the irrepresentable condition when some coefficients are not penalized, which writes as follows: $\exists \eta>0:$
\begin{equation*}
\norm{\left(\hat{\Sigma}_{RS}-\hat{\Sigma}_{RU}\hat{\Sigma}_{UU}^{-1}\hat{\Sigma}_{US} \right)\left(\hat{\Sigma}_{SS}-\hat{\Sigma}_{SU}\hat{\Sigma}_{UU}^{-1}\hat{\Sigma}_{US} \right)^{-1}z^0_S}_{\infty}\le 1-\eta,
\end{equation*}
where $z^0_S=\text{sgn}(\beta^0_S)$.

Note that $\Pi_U=\frac{1}{n}X_U \hat{\Sigma}_{UU}^{-1}X_U'$ and we obtain the above expression.

\end{remark}

Thus, by using the notation of the remark above  we let $U=\{1\}$, $S=S_0$ and $R=[n]\setminus (S_0\cup \{1\})$.

\begin{lemma}\label{l71}
We have that 

\begin{equation*}
\norm{X_R'X_{\{1\}\cup  {S_0}}(X_{\{1\}\cup  {S_0}}'X_{\{1\}\cup  {S_0}})^{-1} z^0_{\{1\}\cup  {S_0}}}_{\infty}= \norm{X_R'A_1 X_{  {S_0}}(X_{  {S_0}}'A_1X_{  {S_0}})^{-1} z^0_ {S_0}}_{\infty}.
\end{equation*}
\end{lemma}

\begin{proof}[Proof of Lemma \ref{l71}]
See Appendix \ref{appA}.
\end{proof}

This means that for tree graphs the irrepresentable condition can be checked for the  ``active set'' $\{1 \} \cup  {S_0}$ instead of $ {S_0}$, but then the first column has to be neglected. This fact is justified, however in a different way then the one we propose, in \cite{qian16} as well.

\begin{remark}[The irrepresentable condition for asymptotic pattern recovery of a signal on a graph does not depend on the orientation of the edges of the graph]

We assume the linear model $Y=f^0+\epsilon, \epsilon\sim\N_n(0,\sigma^2\text{I}_n)$. Then the Edge Lasso can be written as
\begin{equation*}
\hat{f}=\arg\min_{f\in\R^n}\left\{\norm{Y-f}^2_n+2\lambda\norm{(\tilde{I}\tilde{D}f)_{-1}}_1 \right\},
\end{equation*}
where
\begin{equation*}
\tilde{I}\in\mathcal{I}=\left\{\tilde{I}\in\R^n, \tilde{I} \text{ diagonal}, \text{diag}(\tilde{I})\in\{1,-1\}^n \right\}.
\end{equation*}
Define $\beta = \tilde{I}\tilde{D}f$. Then $f=X\tilde{I}\beta$. The linear model assumed becomes $Y=X\tilde{I}\beta^0+\epsilon$ and the estimator
\begin{equation*}
\hat{\beta}=\arg\min_{\beta\in\R^n}\left\{\norm{Y-X\tilde{I}\beta}^2_n+2\lambda \norm{\beta_{-1}}_1 \right\}, \tilde{I}\in\mathcal{I}.
\end{equation*}
It is clear that the now the design matrix is $X\tilde{I}$.
Let us write, without loss of generality
\begin{equation*}
\tilde{I}=\begin{pmatrix} \tilde{I}_{\{1\}\cup S_0} & 0 \\
0 & \tilde{I}_{-(\{1\}\cup S_0)}\end{pmatrix}.
\end{equation*}
According to the Lemma \ref{l71} we can check if $\exists \eta \in (0,1]$:
\begin{equation*}
\norm{\tilde{I}_{-(\{1\}\cup S_0)} X'_{-(\{1\}\cup S_0)} (X'_{\{1\}\cup S_0}X_{\{1\}\cup S_0})^{-1}\tilde{I}_{\{1\}\cup S_0}\tilde{z}^0_{\{1\}\cup S_0}}_{\infty}\le 1-\eta,
\end{equation*}
where $\tilde{z}^0_{\{1\}\cup S_0}= \begin{pmatrix}0 \\ \tilde{z}^0_{S_0} \end{pmatrix}$ and $\tilde{z}^0_{S_0}=\text{sgn}(\beta^0_{S_0})= \tilde{I}_{S_0}\text{sgn}(\tilde{D}f^0)= \tilde{I}_{S_0}\text{sgn}(\bar{\beta}^0)$, where $\bar{\beta}^0=\tilde{D}f^0$, i.e. the vector of truly nonzero jumps when the root has sign $+1$ and the edges are oriented away from it.

Note that $\tilde{I}_{-(\{1\}\cup S_0)}$ does not change the $\ell^{\infty}$-norm and by inserting the expression for $\tilde{z}^0_{\{1\}\cup S_0}$ we get
\begin{equation*}
\norm{ \tilde{I}_{-(\{1\}\cup S_0)} X'_{-(\{1\}\cup S_0)} (X'_{\{1\}\cup S_0}X_{\{1\}\cup S_0})^{-1}\tilde{I}_{\{1\}\cup S_0} \begin{pmatrix}0 & \\ & \tilde{I}_{S_0} \end{pmatrix}\begin{pmatrix}0 \\ \bar{z}^0_{S_0} \end{pmatrix}}_{\infty}\le 1-\eta, \forall \tilde{I}\in \mathcal{I},
\end{equation*}
where $\bar{z}^0_{S_0}=\text{sgn}(\bar{\beta}^0)$.
This means that it is enough to check that $\exists \eta >0$:
\begin{equation*}
\norm{ X'_{-(\{1\}\cup S_0)} (X'_{\{1\}\cup S_0}X_{\{1\}\cup S_0})^{-1}\begin{pmatrix}0 \\ \bar{z}^0_{S_0} \end{pmatrix}}_{\infty}\le 1-\eta, \forall \tilde{I}\in \mathcal{I}
\end{equation*}
to know, for all the orientations of the graph, whether the irrepresentable condition holds. The intuition behind this is that, by choosing the orientation of the edges of the graph, we choose at the same time the sign that the true jumps have across the edges.
\end{remark}

\subsection{Irrepresentable condition for the path graph}

\begin{theorem}[Irrepresentable condition for the transformed Fused Lasso, Theorem 2 in \cite{qian16}]\label{t2qian16}
Consider the model for a piecewise constant signal and let $S_0$ denote the set of indices of the jumps in the true signal, i.e.
\begin{equation*}
S_0=\left\{j: f^0_j\not= f^0_{j-1}, j=2,\cdots, n \right\}=\left\{i_1,\cdots,i_{s_0} \right\},
\end{equation*}
with $s_0=\abs{S_0}$ denoting its cardinality.
The irrepresentable condition for the Edge Lasso on the path graph holds if and only if one of the two following conditions hold:
\begin{itemize}
\item The jump points are consecutive,\\
i.e. $s_0=1$ or $\max_{2\leq k\leq s_0}(i_k-i_{k-1})=1$.
\item All the jumps between constant signal blocks have alternating signs, i.e. 
\begin{equation*}
(f^0_{i_k}-f^0_{i_{k}-1})(f^0_{i_{k+1}}-f^0_{i_{k+1}-1})<0, k=2,\cdots, s_0-1.
\end{equation*}
\end{itemize}
\end{theorem}

\begin{remark}
This fact can as well be easily read out from the consideration made in Section \ref{sec3} and in particular in Lemma \ref{l61}.
\end{remark}

\subsection{Irrepresentable condition for the path graph with one branch}

\begin{Coro}[Irrepresentable condition for the branched path graph]\label{l44}
Assume $S_0\not=0$.
The irrepresentable condition for the branched path graph is satisfied if and only if one of the following cases holds,
\begin{itemize}
\item $s_0=n-1 $ or $s_0=1 $;
\item$\text{sgn}(\beta^0_{i_{s_1}})= -\text{sgn}(\beta^0_{j_{1}})= -\text{sgn}(\beta^0_{k_{1}})$ and in the subvectors $\beta^0_{1:n_1}$ and $\beta^0_{(b,n_1+1:n)}$ there are no two consecutive nonzero entries of $\beta^0$ with the same sign being separated by some zero entry.
\end{itemize}

Note that:
\begin{itemize}
\item If $i_{s_1}=b$, then the requirement above is relaxed to $\text{sgn}(\beta^0_{j_{1}})= \text{sgn}(\beta^0_{k_{1}})$;
\item If $j_1=b+1$, then the requirement above is relaxed to $\text{sgn}(\beta^0_{i_{s_1}})= -\text{sgn}(\beta^0_{k_{1}})$;
\item If $k_{1}=n_1+1$, then the requirement above is relaxed to $\text{sgn}(\beta^0_{i_{s_1}})= -\text{sgn}(\beta^0_{j_{1}})$.
\end{itemize}
\end{Coro}

\begin{proof}[Proof of Lemma \ref{l44}]
This is a special case of Theorem \ref{t72} and follows directly form it.
\end{proof}

\subsection{The irrepresentable condition for general branching points}

When the graph $\mathcal{G}$ has a branching point where arbitrarily many branches are attached, for the irrepresentable condition to be satisfied it is required, in addition to the absence of staircase patterns along the path graphs building $\mathcal{G}$, that the last jump in the path graph containing the branching point has sign $+$ (resp. $-$) and all the jumps in the other path graphs glued to this branching point have sign $-$ (resp. $+$), with respect to the orientation of the edges away from the root. For the index of the $K+1$ jumps surrounding the ramification point we use the same notation as in Subesction \ref{s32}, i.e we denote them by $\{j_1, \ldots, j_{K+1} \}$.

\begin{theorem}\label{t72}
Consider the Edge Lasso estimator on a general ``large enough'' tree graph. The irrepresentable condition for the corresponding (almost) ordinary Lasso problem is satisfied if and only if for the path connecting branching points the conditions of Theorem \ref{t2qian16} hold and for the true signal around any ramification point involving $K+1$ edges, the jump just before it and the jumps right after it have opposite signs. More formally this last condition writes:
\begin{enumerate}
\item
$\text{sgn}(j_1)\text{sgn}(j_l)<0, \forall l \in \left\{l^*\in \{2, \ldots, K+1\}, b_{l^*}\not= 0 \right\}$
\item
and $\text{sgn}(j_l)\text{sgn}(j_{l'})>0, \forall l, l' \in \left\{l^*\in \{2, \ldots, K+1\}, b_{l^*}\not= 0 \right\}$.
\item
and $b_1-1,b_2, \ldots, b_{K+1}< \frac{2}{K+1} b^*$.
\end{enumerate}
Note that if $b_1=1$, then the condition $\text{sgn}(j_1)\text{sgn}(j_l)<0, \forall l \in \left\{l^*\in \{2, \ldots, K+1\}, b_{l^*}\not= 0 \right\}$ is removed.
\end{theorem}

\begin{proof}[Proof of Theorem \ref{t72}]
See Appendix \ref{appA}.
\end{proof}

\section{Conclusion}\label{sec8}

We refined some details of the approach of \cite{dala17} for proving a sharp oracle inequality for the total variation regularized estimator over the path graph. In particular we decided to follow an approach where a coefficient is left unpenalized and we gave a proof of a lower bound on the compatibility constant which does not use probabilistic arguments.

The key point of this article is that we proved that the approach applied on the path graph can indeed be generalized to a branched graph and further to more general tree graphs. In particular we found a lower bound on the compatibility constant and we generalized the result concerning the irrepresentable condition obtained for the path graph by \cite{qian16}.



\appendix

\section{Proofs of Section \ref{sec1}}\label{appE}

\begin{proof}[Proof of Theorem \ref{t21}]

\underline{\textbf{Deterministic part}}\\

Recall the definition of the estimator
\begin{equation*}
\widehat{\beta}=\arg\min_{\beta\in\R^n}\left\{\norm{Y-X\beta}^2_n+2\lambda\norm{\beta_{-1}}_1 \right\}.
\end{equation*}
The KKT conditions  are
\begin{equation*}
\frac{1}{n}X'(Y-X\widehat{\beta})=\lambda\widehat{z}_{-1}, \widehat{z}_{-1}\in \delta\norm{\widehat{\beta}_{-1}}_1,
\end{equation*}
where $\widehat{z}_{-1}\in\R^n$ is a vector with the first entry equal to zero and the remaining ones equal to the subdifferential of the absolute value of the corresponding entry of $\widehat{\beta}$. Inserting $Y=X\beta^0+\epsilon$ into the KKT conditions and multiplying them once by $\widehat{\beta}$ and once by $\beta$ we obtain
\begin{equation*}
-\frac{1}{n}\widehat{\beta}'X'(X(\widehat{\beta}-\beta^0)-\epsilon)=\lambda\norm{\widehat{\beta}_{-1}}_1
\end{equation*}
and
\begin{equation*}
-\frac{1}{n}{\beta}'X'(X(\widehat{\beta}-\beta^0)-\epsilon)= \lambda \beta_{-1}'\widehat{z}_{-1} \leq \lambda\norm{{\beta}_{-1}}_1,
\end{equation*}
where the last inequality follows by the dual norm inequality and the fact that $\norm{\widehat{z}_{-1}}_{\infty}\le1$. Subtracting the first inequality from the second we get
\begin{equation*}
\frac{1}{n}(\widehat{\beta}-{\beta})'X'(X(\widehat{\beta}-\beta^0)-\epsilon) \leq \lambda(\norm{{\beta}_{-1}}_1-\norm{\widehat{\beta}_{-1}}_1).
\end{equation*}

Using polarization we obtain
\begin{equation*}
\norm{X(\widehat{\beta}-\beta)}^2_n+\norm{X(\widehat{\beta}-\beta^0)}^2_n\leq \norm{X(\beta-\beta^0)}^2_n+\frac{2}{n}(\widehat{\beta}-\beta)'X'\epsilon+2\lambda\left(\norm{\beta_{-1}}_1-\norm{\widehat{\beta}_{-1}}_1\right).
\end{equation*}
Let $S\subset\{2,\ldots, n\}$. We have that
\begin{eqnarray*}
\norm{\beta_{-1}}_1-\norm{\widehat{\beta}_{-1}}_1&=&\norm{\beta_{S}}_1-\norm{\widehat{\beta}_{S}}_1-\norm{\beta_{-(\{1\}\cup S)}}_1-\norm{\widehat{\beta}_{-(\{1\}\cup S)}}_1\\
&+& 2\norm{\beta_{-(\{1\}\cup S)}}_1\\
&\leq& \norm{\beta_{S}-\widehat{\beta}_{S} }_1-\norm{\beta_{-(\{1\}\cup S)} -\widehat{\beta}_{-(\{1\}\cup S)}}_1+2\norm{\beta_{-(\{1\}\cup S)}}_1.
\end{eqnarray*}
Thus we get the ``basic'' inequality
\begin{eqnarray*}
& &\norm{X(\widehat{\beta}-\beta)}^2_n+\norm{X(\widehat{\beta}-\beta^0)}^2_n \leq \norm{X(\beta-\beta^0)}^2_n+4\lambda\norm{\beta_{-(\{1\}\cup S)}}_1\\
&+& \underbrace{ \frac{2}{n}(\widehat{\beta}-\beta)'X'\epsilon+ 2\lambda\left(\norm{(\beta-\widehat{\beta})_{S} }_1-\norm{(\beta -\widehat{\beta})_{-(\{1\}\cup S)}}_1\right)  }_{\text{I}}.
\end{eqnarray*}
We are going to utilize the approach described by \cite{dala17} to handle the remainder term I with care.  Since $\text{I}_n=\Pi_{\{1\}\cup S}+(\text{I}_n-\Pi_{\{1\}\cup S})$, it follows that
\begin{equation*}
(\widehat{\beta}-\beta)'X'\epsilon=(\widehat{\beta}-\beta)'X'\Pi_{\{1\}\cup S}\epsilon+(\widehat{\beta}-\beta)_{-(\{1\}\cup S)}'X_{-(\{1\}\cup S)}'(\text{I}_n-\Pi_{\{1\}\cup S}) \epsilon.
\end{equation*}
Indeed the antiprojection of elements of $V_{\{1\}\cup S}$ is zero.
Note that
\begin{equation*}
(\widehat{\beta}-\beta)_{-(\{1\}\cup S)}'X_{-(\{1\}\cup S)}'(\text{I}_n-\Pi_{\{1\}\cup S}) \epsilon\leq\sum_{j\in{-(\{1\}\cup S)} } \abs{\widehat{\beta}-\beta}_j\abs{X'_j(\text{I}_n-\Pi_{\{1\}\cup S})\epsilon}.
\end{equation*}
Restricting ourselves to the set
\begin{equation*}
F=\left\{\abs{X_j'(\text{I}-\Pi_{\{1\}\cup S})\epsilon}\leq \frac{\lambda n}{\gamma}\frac{\norm{X_j'(\text{I}_n-\Pi_{\{1\}\cup S})}_2}{\sqrt{n}}, \forall j\in{-(\{1\}\cup S)}  \right\},
\end{equation*}
 for $\gamma\geq 1$ 
we obtain
\begin{eqnarray*}
\text{I}&\leq& \frac{2}{n}(\widehat{\beta}-\beta)'X'\Pi_{\{1\}\cup S}\epsilon\\
&+& 2\lambda \left( \norm{(\widehat{\beta}-\beta)_S}_1-\norm{(\widehat{\beta}-\beta)_{-(\{1\}\cup S)}}_1+\norm{(\frac{\omega}{\gamma}\odot(\widehat{\beta}-\beta))_{-(\{1\}\cup S)}}_1\right)\\
&\leq&2\frac{\norm{X(\widehat{\beta}-\beta)}_2}{\sqrt{n}}\frac{\norm{\Pi_{\{1\}\cup S}\epsilon}_2}{\sqrt{n}}\\
&+& 2\lambda\left(\norm{(\widehat{\beta}-\beta)_S}_1-\norm{(w\odot(\widehat{\beta}-\beta))_{-(\{1\}\cup S)}}_1\right),
\end{eqnarray*}

Using the definition of the weighted compatibility constant and the convex conjugate inequality we obtain
\begin{eqnarray*}
\text{I}&\leq& 2\frac{\norm{X(\widehat{\beta}-\beta)}_2}{\sqrt{n}}\left(\frac{\norm{\Pi_{\{1\}\cup S}\epsilon}_2}{\sqrt{n}}+\lambda\frac{\sqrt{ s+1}}{\kappa_w(S)}\right)\\
&\leq& \norm{X(\widehat{\beta}-\beta)}^2_n+\left(\frac{\norm{\Pi_{\{1\}\cup S}\epsilon}_2}{\sqrt{n}}+\lambda\frac{\sqrt{s+1}}{\kappa_w(S)}\right)^2
\end{eqnarray*}
We see that $\norm{X(\widehat{\beta}-\beta)}^2_n$ cancels out and we are left with
\begin{eqnarray*}
\norm{X(\widehat{\beta}-\beta^0)}_n^2&\leq& \inf_{\beta\in\R^n } \left\{\norm{X(\beta-\beta^0)}_n^2 + 4\lambda\norm{\beta_{-(\{1\}\cup S)}}_1 \right\}\\
&+& \left(\frac{\norm{\Pi_{\{1\}\cup S}\epsilon}_2}{\sqrt{n}}+\lambda\frac{\sqrt{ s+1}}{\kappa_w(S)}\right)^2.
\end{eqnarray*}

It now remains to find a lower bound for $\pr(F)$ and a high-probability upper bound for $\norm{\Pi_{\{1\}\cup S}\epsilon}^2_n$.\\

\underline{\textbf{Random part}}\\
\begin{itemize}
\item First, we lower bound $\pr(F)$, thanks to the following lemma.

\begin{lemma}[The maximum of $p$ random variables, Lemma 17.5 in \cite{vand16}]\label{l175vand16}
Let $V_1, \ldots, V_{p}$ be real valued random variables. Assume that $\forall j \in \{1,\ldots, p\}$ and $\forall r>0$
\begin{equation*}
\ex\left[e^{r \abs{V_j}} \right]\leq 2 e^{\frac{r^2}{2}}.
\end{equation*}
Then, $\forall t>0$
\begin{equation*}
\pr\left(\max_{1\leq j \leq p} \abs{V_j}\geq \sqrt{2\log (2p)+2t} \right) \leq e^{-t}.
\end{equation*}
\end{lemma}

We now apply Lemma \ref{l175vand16} to $F$. Note that $F$ can be written as
\begin{equation*}
F=\left\{\max_{j\in{-(\{1\}\cup S)}}\left|\frac{X_j'(\text{I}_n-\Pi_{\{1\}\cup S})\epsilon}{\sigma\norm{X_j'(\text{I}_n-\Pi_{\{1\}\cup S})}_2}\right|\leq \frac{\lambda \sqrt{n}}{\gamma\sigma}   \right\}.
\end{equation*}

Since $X_j'(\text{I}-\Pi_{\{1\}\cup S})\epsilon\sim\N(0,\sigma^2 \norm{X_j'(\text{I}_n-\Pi_{\{1\}\cup S})}^2_2)$, we obtain that for $V_1, \ldots, V_{n-s-1}\sim\N(0,1)$
\begin{equation*}
F=\left\{\max_{1\leq j\leq n-s-1}\abs{V_j}\leq \frac{\lambda \sqrt{n}}{\gamma\sigma}\right\}.
\end{equation*}

The moment generating function of $\abs{V_j}$ is
\begin{equation*}
\ex\left[e^{r\abs{Vj}} \right]=2(1-\Phi(-r))e^{\frac{r^2}{2}}\leq 2 e^{\frac{r^2}{2}}, \forall r>0
\end{equation*}

Choosing, for some $\delta\in (0,1)$, $\lambda={\gamma\sigma }\sqrt{2\log\left({4(n-s-1)}/{\delta}\right)/n}$ and applying Lemma \ref{l175vand16} with $p=n-s-1$ and $t=\log \left(\frac{2}{\delta} \right)$, we obtain
\begin{equation*}
\pr(F)\geq 1-{\delta}/{2}.
\end{equation*}

\item
Second, we are going to find an high probability upper bound for 
\begin{equation*}
\norm{\Pi_{\{1\}\cup S}\epsilon}^2_n=\frac{\sigma^2}{n}\underbrace{\frac{\norm{\Pi_{\{1\}\cup S}\epsilon}^2_2}{\sigma^2}}_{\sim\chi^2_{s+1}},
\end{equation*}
where $\text{rank}(\Pi_{\{1\}\cup S})=s+1$. We use Lemma 8.6 in \cite{vand16}, which reproves part of Lemma 1 in \cite{laur00}.

\begin{lemma}[The special case of $\chi^2$ random variables, Lemma 1 in \cite{laur00},Lemma 8.6 in \cite{vand16}]\label{l86vand16}
Let $X\sim\chi^2_d$. Then, $\forall t>0$
\begin{equation*}
\pr\left(X\geq d+2\sqrt{dt}+2t \right)\leq e^{-t}
\end{equation*}
\end{lemma}

Note that from Lemma \ref{l86vand16} it follows that
\begin{equation*}
\pr\left(\sqrt{X}\leq \sqrt{d}+\sqrt{2t} \right)\geq \pr\left(X\leq d+2\sqrt{dt}+2t \right)\geq 1-e^{-t}
\end{equation*}

Define
\begin{equation*}
G:=\left\{ \frac{\norm{\Pi_{\{1\}\cup S}\epsilon}_2}{\sqrt{n}} \leq \sqrt{\frac{\sigma^2}{n}}\left(\sqrt{s+1}+\sqrt{2\log\left({2}/{\delta} \right)}\right) \right\}.
\end{equation*}
By applying Lemma \ref{l86vand16} with $t=\log\left({2}/{\delta} \right)$, for some $\delta\in(0,1)$, we get
\begin{equation*}
\pr(G)\geq 1-{\delta}/{2}.
\end{equation*}

\end{itemize}

If we choose $\lambda={\gamma\sigma}\sqrt{2\log\left({4(n-s-1)}/{\delta} \right)/n}$ and apply twice the inequality $(a+b)^2\leq 2(a^2+b^2)$, we get that with probability $\pr(F\cap G)\geq 1-\delta$ the following oracle inequality holds
\begin{eqnarray*}
\norm{X(\widehat{\beta}-\beta^0)}_n^2&\leq& \inf_{\beta\in\R^n } \left\{\norm{X(\beta-\beta^0)}_n^2 + 4\lambda\norm{\beta_{-(\{1\}\cup S)}}_1 \right\}\\
&+&\frac{4\sigma^2}{n}\left((s+1)+2\log\left({2}/{\delta}\right)+\frac{\gamma^2 (s+1)}{\kappa^2_w(S)} \log \left({4(n-s-1)}/{\delta}  \right) \right).
\end{eqnarray*}
The statement of the lemma is obtained using the identity $f=X\beta$.
\end{proof}

\section{Proofs of Section \ref{sec4}}\label{appB}

Let $f\in\R^n$ be a function defined at every vertex of a connected nondegenerate graph $\mathcal{G}$. Moreover let 
\begin{equation*}
f_{(n)}\geq \ldots \geq f_{(1)}
\end{equation*}
be an ordering of $f$, with arbitrary order within tuples. Let $D$ denote the incidence matrix of the graph $\mathcal{G}$.

\begin{lemma}[Lemma 11.9 in \cite{vand18}]\label{lab1}
It holds that
\begin{equation*}
\norm{Df}_1\geq f_{(n)}-f_{(1)}.
\end{equation*}
\end{lemma}

\begin{remark}
For the special case of $\mathcal{G}$ being the path graph, we have equality in  Lemma \ref{lab1} when $f$ is nonincreasing or nondecreasing on the graph.
\end{remark}

\begin{proof}[Proof of Lemma \ref{lab1}]
Since $\mathcal{G}$ is connected there is a path between any two vertices.
Therefore there is a path connectiong the vertices where $f$ takes the values $f_{(n)}$ and $f_{(1)}$.
The total variation of a function defined on a graph is nondecreasing in the number of edges of the graph.
Let us now consider $f_P$, the restriction of $f$ on a path $P$ connecting $f_{(1)}$ to $f_{(n)}$.
If $f$ is nondecreasing on the path $P$, then $\norm{Df_P}_1=f_{(n)}-f_{(1)}$, otherwise $\norm{Df_P}_1\geq f_{(n)}-f_{(1)}$.
Since $\mathcal{G}$ has at least as many edges as $P$:
\begin{equation*}
\norm{Df}_1\geq \norm{Df_P}_1\geq f_{(n)}-f_{(1)}.
\end{equation*}
\end{proof}

\begin{lemma}[Lemma 11.10 in \cite{vand18}]\label{lab2}
It holds for any $j \in \{1 , \ldots , n\} $ that
  $$ f_j - \norm{Df}_1 \le f_{(1)} \le {1 \over n} \sum_{i=1}^n |f_i| , $$
  and
  $$ - f_j - \norm{Df}_1  \le - f_{(n)} \le {1 \over n} \sum_{i=1}^n |f_i| . $$
\end{lemma}

\begin{proof}[Proof of Lemma \ref{lab2}]
See \cite{vand18}.
\end{proof}

\begin{lemma}[Lemma 11.11 in \cite{vand18}]\label{lab3}
Let $f \in \R^n$ be defined over a connected graph $\mathcal{G}_f$ whose incidence matrix is $D_f$. The total variation of $f$ is $\norm{D_f f}_1$.
Analogously, let $g \in \R^m$ be defined over a connected graph $\mathcal{G}_g$ whose incidence matrix is $D_g$. The total variation of $g$ is $\norm{D_g f}_1$.
Then for any $j \in \{1 , \ldots , n\}$ and
  $k \in \{1 , \ldots, m \}$
  $$ | f_j - g_k | - \norm{D_f f}_1 - \norm{D_g g}_1 \le {1 \over n} \sum_{i=1}^n | f_i | +
  {1 \over m} \sum_{i=1}^m | g_i | . $$
\end{lemma}

\begin{proof}[Proof of Lemma \ref{lab3}]
Suppose without loss of generality that $f_j \ge  g_k$. Then by Lemma \ref{lab2}
  \begin{eqnarray*}
   | f_j - g_k | - \norm{D_f f}_1 - \norm{D_g g}_1&=& \underbrace{(f_j - \norm{D_f f}_1 )}_{\le \sum_{i=1}^n | f_i | /n }  +
   \underbrace{(- g_k - \norm{D_g g}_1)}_{\le \sum_{i=1}^m | g_i | / m }  \\
   & \le & {1 \over n} \sum_{i=1}^n | f_i | +
  {1 \over m} \sum_{i=1}^m | g_i | .
    \end{eqnarray*}
\end{proof}

\begin{proof}[Proof of Lemma \ref{l31}]
See proof of Theorem 6.1 in \cite{vand18} and Appendix \ref{int} for an intuition.
\end{proof}

\begin{proof}[Proof of Corollary \ref{c32}]
See proof of Theorem 6.1 in \cite{vand18}.
%
%
\end{proof}

\begin{proof}[Proof of Lemma \ref{l33}]
See proof of Lemma 9.1 in \cite{vand18} and Appendix \ref{int} for an intuition.
\end{proof}

\begin{proof}[Proof of Corollary \ref{c31}]

Let $A=\text{I}-\Pi_{\{ 1\}\cup S}$ denote the antiprojection matrix on the coulmns of $X$ indexed by $\{ 1\}\cup S$. By using the definition of $w_i$ and $\omega_i$, we have that

\begin{eqnarray*}
\norm{Dw}^2_2=\sum_{i=2}^{n}(w_i-w_{i-1})^2&=& \frac{1}{\gamma^2}\sum_{i=2}^n(\omega_i-\omega_{i-1})^2\\
&=& \frac{1}{\gamma^2 n} \sum_{i=2}^n (\norm{X_i'A}_2-\norm{X_{i-1}'A}_2 )^2.
\end{eqnarray*}

For the path graph we have, thanks to Section \ref{sec3},
\begin{eqnarray*}
\sum_{i=2}^n (w_i-w_{i-1})^2&=& 
\frac{1}{n\gamma^2}\sum_{j=1}^{s+1}\sum_{i=1}^{b_j}\frac{(\sqrt{i(b_j-i)}-\sqrt{(i-1)(b_j-(i-1))})^2}{b_j}\\
&=&  \frac{1}{n\gamma^2}\sum_{j=1}^{s+1}\sum_{i=1}^{b_j}\frac{({i(b_j-i)}-{(i-1)(b_j-(i-1))})^2}{b_j(\sqrt{i(b_j-i)}+\sqrt{(i-1)(b_j-(i-1))})^2}\\
&\le& \frac{1}{n\gamma^2}\sum_{j=1}^{s+1}\sum_{i=1}^{b_j} \frac{(-2i+b_j+1)^2}{b_j(-2i^2+2(b_j+1)i-(b_j+1))}\\
&\leq& \frac{1}{n\gamma^2}\sum_{j=1}^{s+1} b_j\sum_{i=1}^{b_j}\frac{1}{f(i)}.
\end{eqnarray*}

Consider now the function $f(x)=-2x^2+2(c+1)x-(c+1)$, with $c>0$ a constant. This function is strictly concave, has two zeroes at $x=\frac{c+1}{2}\pm \frac{\sqrt{c^2-1}}{2}$ and a global maximum at $x=\frac{c+1}{2}$. We note that $\forall c\geq 1$ the left zero point $\frac{c+1}{2}- \frac{\sqrt{c^2-1}}{2}\leq 1$. Moreover, $\forall x\in [1, c/2]$, $f(x)\geq cx$. Using the symmetry of quadratic functions around the global maximum we obtain, in partial analogy to \cite{dala17},
\begin{eqnarray*}
\norm{Dw}^2_2&=&\sum_{i=2}^{n}(w_i-w_{i-1})^2\le  \frac{1}{\gamma^2 n}\sum_{j=1}^{s+1}b_j\left(\sum_{i=1}^{\lceil b_j/2\rceil}\frac{1}{b_j i}+\sum_{i=1}^{\lfloor b_j/2\rfloor}\frac{1}{b_j i}\right)\\
&\leq& \frac{5}{2\gamma^2 n} \sum_{j=1}^{s+1} \log(\lceil b_j/2\rceil \lfloor b_j/2\rfloor)= \frac{5}{2\gamma^2 n} \log\left(\prod_{i=1}^{2(s+1)}\abs{\Delta}_i \right)\\
&=& \frac{5}{\gamma^2 n} (s+1) \log(\bar{\abs{\Delta}})\le \frac{5}{\gamma^2 n} (s+1) \log(n/(2s+2))\\
&\le& \frac{5}{\gamma^2 n} (s+1) \log(n/(s+1)),
\end{eqnarray*}

where $\bar{\abs{\Delta}}$ is the geometric mean of $\abs{\Delta}$, which is upper bounded by the arithmetic mean of $\abs{\Delta}$, which is $n/(2s+2)$. Moreover the constant $5/2$ and the assumption $b_i\ge 4,\forall i>1$ come from the fact that 
$$
 \frac{\sum_{i=1}^k i^{-1}}{\log i}
$$
is finite only if $i\ge 2$, is decreasing in $i$ and has value approximately 2.16 when $i=2$.
Moreover the vector $\abs{\Delta}\in\R^{2s+2}$ is defined as
$$
\abs{\Delta}\in\R^{2s+2}=\left(  \lfloor{d_1}/{2}\rfloor,\lceil{d_1}/{2}\rceil, \ldots,\lfloor{d_{s+1}}/{2}\rfloor,\lceil{d_{s+1}}/{2}\rceil\right).
$$

We now have to find an upper bound for $K$. Since the choice of $u_j$ is arbitrary, we choose $u_j=\lfloor d_j/2 \rfloor, j\in \{2, \ldots, s \}$, which minimize the upper bound among the integers. We thus have that $K\leq \frac{2s}{\bar{\Delta}_h}$,
where $\bar{\Delta}_h$ is the harmonic mean of $\Delta$.

Finaly, for the path graph we have
\begin{eqnarray*}
\frac{1}{\kappa^2_w(S)}&\le& \frac{2n}{s+1}(K+\norm{Dw}^2_2)\\
&\le& \frac{2n}{\gamma^2(s+1)}\left(\frac{2\gamma^2 s}{\bar{\Delta}_h}+5\frac{s+1}{n}\log(n/(s+1))\right),
\end{eqnarray*}

and we obtain the Corollary \ref{c31}.

\end{proof}

\subsection{Outline of proofs by means of a minimal toy example}\label{int}

For giving an intuition to the reader we present a minimal toy example. Consider the path graph with $n=8$ and let $S=\{3,7 \}$. In this example $d_1=2, d_2=4, u_2=2,d_3=2$. We write
\begin{eqnarray*}
\norm{\beta_S}_1-\norm{\beta_{-(\{1\}\cup S)}}_1 &=& \abs{f_3-f_2}-\abs{f_2-f_1}-\abs{f_4-f_3}\\
&& + \abs{f_7-f_6}-\abs{f_6-f_5}-\abs{f_8-f_7}\\
&& - \abs{f_5-f_4}
\end{eqnarray*}

The idea now is to apply Lemma \ref{lab3} twice, once to the path graphs $(\{1,2\},(1,2))$ and $(\{3,4\},(3,4))$ and once to the path graphs $(\{5,6\},(5,6))$ and $(\{7,8\},(7,8))$. Note that the term $\abs{f_5-f_4}$ is not needed to apply Lemma \ref{lab3} and thus can be left away. We get
$$
\norm{\beta_S}_1-\norm{\beta_{-(\{1\}\cup S)}}_1\le \frac{1}{2}\sum_{i=1}^8 \abs{f_i}\le \sqrt{2}\norm{f}_2,
$$
where the last step follows by the Cauchy-Schwarz inequality. We thus see that we can handle graphs built by modules consisting of small path graphs containing an edge in $S$ and at least one vertex not involved in this edge on each side. The edges connecting these modules can then be neglected when upperbounding $\norm{\beta_S}_1-\norm{\beta_{-(\{1\}\cup S)}}_1$.

In the weighted case we define $g_i=w_if_i, i=1, \ldots, 8$ and write
\begin{eqnarray*}
&&\norm{(w\odot\beta)_S}_1-\norm{(w\odot\beta)_{-(\{1\}\cup S)}}_1 \\
&&\le  w_3\abs{f_3-f_2}-w_2\abs{f_2-f_1}-w_4\abs{f_4-f_3}\\
&& + w_7\abs{f_7-f_6}-w_6\abs{f_6-f_5}-w_8\abs{f_8-f_7}\\
&& \le \abs{g_3-g_2}-\abs{g_2-g_1}-\abs{g_4-g_3}\\
&& + \abs{g_7-g_6}-\abs{g_6-g_5}-\abs{g_8-g_7}\\
&& + \sum_{i=2}^4 \abs{w_i-w_{i-1}}\abs{f_{i-1}}+ \sum_{i=6}^8 \abs{w_i-w_{i-1}}\abs{f_{i-1}}\\
&& \le\sqrt{1/4 \norm{w}^2_2}\norm{f}_2\\
&& + \sqrt{ \sum_{i=2}^4 (w_i-w_{i-1})^2+ \sum_{i=6}^8 (w_i-w_{i-1})^2}\sqrt{\sum_{i=1}^3f^2_i+\sum_{i=5}^7f^2_i}\\
&& \le \left( \sqrt{2}\norm{w}_{\infty}+  \sqrt{ \sum_{i=2}^4 (w_i-w_{i-1})^2+ \sum_{i=6}^8 (w_i-w_{i-1})^2}\right) \norm{f}_2.
\end{eqnarray*}
Here as well, note that the squared difference of the weights across the edge connecting the two modules (smaller but large enough path graphs containing an element of $S$) can be neglected. The procedure exemplified here can be used to handle larger tree graphs, as long as one is able to decompose them in such smaller modules. The fact that squared weights differences can be neglected at the junction of modules will be of use in the proof of Corollary \ref{c41}.

\begin{remark}
The limits of this approach are given by Lemma \ref{lab3}, since its use requires the presence of at least a distinct edge not in $S$ on the left and on the right for each edge in $S$ not sharing vertices with edges used to handle other elements of $S$. Thus $s\leq n/4 $. However, this limitation  is very likely to be of scarce relevance if some kind of minimal length condition holds, see for instance \cite{dala17,gunt17,lin17b}.
\end{remark}
\section{Proofs of Section \ref{sec5}}\label{appC}

\begin{proof}[Proof of Lemma \ref{l41}]
The result follows directly by the proof of Lemma \ref{l31} (i.e. Theorem 6.1 in \cite{vand18}), by the decomposition of the branched path graph into three path graphs. See Appendix \ref{int} for an intuition.
\end{proof}

\begin{proof}[Proof of Corollary \ref{l43}]
The proof follows by direct calculations in analogy to the one of Corollary \ref{c32} (i.e. Theorem 6.1in \cite{vand18}).
\end{proof}

\begin{proof}[Proof of Lemma \ref{l42}]
In the proof of Lemma \ref{l31} and Lemma \ref{l33} (i.e. Theorem 6.1 and Lemma 9.1 in \cite{vand18}) and in Appendix \ref{int} it is made clear, that the use of Lemma \ref{lab3} requires that the edges connecting the smaller pieces into which the path graph is partitioned are taken out of consideration when upper bounding $\norm{\beta_{S}}_1-\norm{\beta _{-(\{1\}\cup S)}}_1$ resp. $\norm{(\beta \odot w)_{S}}_1-\norm{(\beta \odot w)_{-(\{1\}\cup S)}}_1$. This results in an upper bound containing only the square of some of the consecutive pairwise differences between the entries of $w$, the vector of weights. This ``incomplete'' sum can then of course be upper bounded by $\norm{Dw}_2$, where $D$ is the incidence matrix of the path graph.


In the case of the branched path graph the same reasoning applies in particular to the two edges connecting together the three path graphs defined by the second notation. Indeed these can be left away and it is natural to do so. Thus, in full analogy to the procedure exposed in the proofs of Lemma \ref{l31} and \ref{l33} (i.e. Theorem 6.1 and Lemma 9.1 in \cite{vand18}) for the path graph, the statement of Lemma \ref{l42} follows. See Appendix \ref{int} for an intuition
\end{proof}

\begin{proof}[Proof of Corollary \ref{c41}]
We use the calculations  done in Section \ref{sec3}.
By writing

\begin{equation*}
a_i=\sqrt{\frac{(b_j-i)i}{b_j}}, i\in \{0,1, \ldots, b_j \}, j\in [s+3]\setminus \{s_1+1, s_1+2, s_1+s_2+3 \},
\end{equation*}

\begin{equation*}
a^*_i=\sqrt{\frac{(b^*-i)i}{b^*}},i\in \{0,1, \ldots, b^* \} \text{ where } b^*=b_{s_1+1}+b_{ s_1+2}+b_{ s_1+s_2+3}
\end{equation*}

we obtain that

\begin{eqnarray*}
\norm{Dw}^2_2&=&\frac{1}{\gamma^2 n}\left\{\sum_{j=1}^{s_1} \sum_{i=1}^{b_j}(a_i-a_{i-1})^2+\sum_{j=s_1+3}^{s_1+s_2+2} \sum_{i=1}^{b_j}(a_i-a_{i-1})^2\right. \\
&+&\sum_{j=s_1+s_2+4}^{s_1+s_2+s_3+3} \sum_{i=1}^{b_j}(a_i-a_{i-1})^2  \\
&+& \sum_{i=1}^{b_{s_1+1}-1}(a^*_i-a^*_{i-1})^2+\sum_{i=1}^{b_{s_1+2}}(a^*_i-a^*_{i-1})^2+\sum_{i=1}^{b_{s_1+s_2+3}}(a^*_i-a^*_{i-1})^2 \\
&+&\left. (a^*_{b_{s_1+1}-1}-a^*_{b_{s_1+2}})^2+(a^*_{b_{s_1+1}-1}-a^*_{b_{s_1+s_2+3}})^2   \right\}.
\end{eqnarray*}

Indeed we can bound all the terms except the last two ones by applying the reasoning developed for the path graph. 

We are now interested in upper bounding $\norm{D^*w}^2_2$ rather $\norm{Dw}^2_2$.
We have that

\begin{eqnarray*}
\norm{D^*w}^2_2&=&\frac{1}{\gamma^2 n}\left\{\sum_{j=1}^{s_1} \sum_{i=1}^{b_j}(a_i-a_{i-1})^2+\sum_{j=s_1+3}^{s_1+s_2+2} \sum_{i=1}^{b_j}(a_i-a_{i-1})^2\right. \\
&+&\left. \sum_{j=s_1+s_2+4}^{s_1+s_2+s_3+3} \sum_{i=1}^{b_j}(a_i-a_{i-1})^2+ z\right\},
\end{eqnarray*}

where we distinguish the following four cases

\begin{enumerate}[1)]
\item
\begin{eqnarray*}
z&=& \sum_{i=1}^{b_{s_1+1}-1}(a^*_i-a^*_{i-1})^2+\sum_{i=1}^{b_{s_1+2}}(a^*_i-a^*_{i-1})^2+\sum_{i=1}^{b_{s_1+s_2+3}}(a^*_i-a^*_{i-1})^2 \\
&\le & 5/2 \log (\lfloor b^*/3\rfloor\lceil b^*/3\rceil(b^*-\lfloor b^*/3\rfloor-\lceil b^*/3\rceil) ).
\end{eqnarray*}

\item
Assume without loss of generality that $b_{s_1+s_2+3}=0$.
\begin{eqnarray*}
z&=& \sum_{i=1}^{b_{s_1+1}-1}(a^*_i-a^*_{i-1})^2+\sum_{i=1}^{b_{s_1+2}}(a^*_i-a^*_{i-1})^2+(a^*_{b_{s_1+1}-1})^2\\
&\le & \sum_{i=1}^{b^*}(a^*_i-a^*_{i-1})^2+ \max_{i\in[b^*]} (a^*_i)^2\\
&\le& 5/2 \log (\lfloor b^*/2\rfloor\lceil b^*/2\rceil) +b^*/2
\end{eqnarray*}

\item
\begin{enumerate}[a)]
\item Assume without loss of generality that $b_{s_1+s_2+3}=2$.
\begin{eqnarray*}
z&\le& \sum_{i=1}^{b_{s_1+2}}(a^*_i-a^*_{i-1})^2+\sum_{i=1}^{b_{s_1+s_2+3}}(a^*_i-a^*_{i-1})^2+ (a^*_{b^*-3})^2\\
&\le& 5/2 \log (\lfloor b^*/2\rfloor\lceil b^*/2\rceil )+3
\end{eqnarray*}

\item We have the choice, which edge we can leave out of our consideration: either the edge $(b,b+1)$ of the edge $(b,n_1+1)$. In both cases
\begin{equation*}
 z\le\sum_{i=1}^{b_{s_1+2}}(a^*_i-a^*_{i-1})^2+\sum_{i=1}^{b_{s_1+s_2+3}}(a^*_i-a^*_{i-1})^2 +[(a^*_{b_{s_1+2}})^2\wedge(a^*_{b_{s_1+s_2+3}})^2].
\end{equation*}
Denote $y:=b_{s_1+2}$. Then $b_{s_1+s_2+3}=b^*-1-y$.
We get that
\begin{equation*}
(a^*_y)^2\wedge(a^*_{b-y-1})^2=\begin{cases}(a^*_y)^2 &, 3\le y \le (b^*-1)/2\\
(a^*_{b-y-1})^2 &, (b^*-1)/2\le w \le b^*-3,
\end{cases} \le b^*/4.
\end{equation*}
Thus
\begin{equation*}
z\le 5/2 \log (\lfloor b^*/2\rfloor\lceil b^*/2\rceil ) + b^*/4
\end{equation*}

\end{enumerate}

\item
Assume without loss of generality $b_{s_1+s_2+3}=1$, then
\begin{eqnarray*}
z&\le & \sum_{i=1}^{b_{s_1+1}-1}(a^*_i-a^*_{i-1})^2+\sum_{i=1}^{b_{s_1+2}}(a^*_i-a^*_{i-1})^2+\sum_{i=1}^{1}(a^*_i-a^*_{i-1})^2 \\
&+&  (a^*_1-a^*_{b_{s_1+1}-1} )^2.
\end{eqnarray*}
Let $x:= b_{s_1+1}$. We have that
\begin{eqnarray*}
& &\max_{3\le x \le b^*-3}\left(\sqrt{\frac{b^*-1}{b^*}}-\sqrt{\frac{(b^*-x+1)(x-1)}{b^*}} \right)^2\\
&=& \frac{1}{b^*} (b^*/2-\sqrt{b^*-1})^2\le b^*/4,
\end{eqnarray*}
where the maximum is attained at $x=(b^*+2)/2$ and the last inequality holds since $b^*/2\ge \sqrt{b^*-1},\forall b^*\ge 1$.
Therefore
$$
z\le 5/2 \log (\lfloor b^*/3\rfloor\lceil b^*/3\rceil(b^*-\lfloor b^*/3\rfloor-\lceil b^*/3\rceil))+ b^*/4
$$
\end{enumerate}
Now define the vectors
$$
\abs{\Delta}^i:=\begin{cases}
\left(\lfloor d^i_1/2 \rfloor ,\lceil d^i_1/2 \rceil, \ldots,\lfloor d^i_{s_i}/2 \rfloor,\lceil d^i_{s_i}/2 \rceil , \delta^i \right),i=1\\
\left( \delta^i,\lfloor d^i_2/2 \rfloor ,\lceil d^i_2/2 \rceil, \ldots,\lfloor d^i_{s_i+1}/2 \rfloor,\lceil d^i_{s_i+1}/2 \rceil\right),i=2,3
\end{cases}, \in \R^{2s_i+1}.
$$
We can distinguish the following four cases:

\begin{enumerate}[1)]
\item $(\delta^1,\delta^2,\delta^3)=(\lfloor b^*/3\rfloor,\lceil b^*/3\rceil,b^*-\lfloor b^*/3\rfloor-\lceil b^*/3\rceil)$ in any order;

\item $\delta^2=1$ or $\delta^3=1$ and the nonzero $\delta$'s take values $\lfloor b^*/2\rfloor$ and $\lceil b^*/2\rceil$;

\item See Case 2), however with $\delta^1=1$;

\item See Case 1).

\end{enumerate}

Let $\abs{\Delta}:=(\abs{\Delta}^1,\abs{\Delta}^2,\abs{\Delta}^3)\in\R^{2s+3}$.

In analogy to the case of the path graph, see Proof of Corollary \ref{c31} in Appendix \ref{appB}, we can find the bound
\begin{eqnarray*}
\norm{D^*w}^2_2&\le& (5/2)\log\left(\prod_{i=1}^{2s+3}\abs{\Delta}_i \right)+ \zeta\\
&\le& (5/2) (2s+3)\log\left(\frac{n+1}{2s+3} \right)+ \zeta,
\end{eqnarray*}
where
$$
\zeta=\begin{cases}
0 &, \text{ Case 1)}\\
b^*/2 &,\text{ Case 2)}\\
3 &,\text{ Case 3)a)} \\
b^*/4 &,\text{ Case 3)b)}\\
b^*/4 &,\text{ Case 4)}
\end{cases}.
$$

For the compatibility constant we have that $K_b\leq \frac{2s}{\bar{\Delta}_h}$ and we obtain an upper bound for the reciprocal of the weighted compatibility constant
\begin{equation*}
\frac{1}{\kappa^2_w(S)}\leq \frac{2n}{\gamma^2(s+1)}\left(\frac{2\gamma^2 s}{\bar{\Delta}_h}+\frac{5(2s+3)\log (n+1)}{2n}+ \frac{\zeta}{n} \right),
\end{equation*}
where $\zeta$ is as above.
We therefore get Corollary \ref{c41}.
\end{proof}

\section{Proofs of Section \ref{sec7}}\label{appA}

\subsection{Preliminaries}

We will need the following results.
\begin{lemma}[The inverse of a partitioned matrix]\label{invpart}
Let
\begin{equation*}
A=\begin{pmatrix*}[l] A_{11} & A_{12}\\ A_{21} & A_{22}
\end{pmatrix*}
\end{equation*}
where $A_{11}$ and $A_{22}$ are invertible matrices and $A_{11}-A_{12}A_{22}^{-1}A_{21}$ and $A_{22}-A_{21}A_{11}^{-1}A_{12} $ are invertible as well. 
Then
\begin{equation*}
A^{-1}=\begin{pmatrix*}[l]
(A_{11}-A_{12}A_{22}^{-1}A_{21})^{-1} & -(A_{11}-A_{12}A_{22}^{-1}A_{21})^{-1}A_{12}A_{22}^{-1}\\
-(A_{22}-A_{21}A_{11}^{-1}A_{12})^{-1}A_{21}A_{11}^{-1} & (A_{22}-A_{21}A_{11}^{-1}A_{12})^{-1}
\end{pmatrix*}
\end{equation*}
\end{lemma}

\begin{lemma}[The inverse of the sum of two matrices, \cite{mill81}]\label{lmill81}
Let $G$ and $G+E$ be invertible matrices, where $E$ is a matrix of rank one.
Let $g:=\text{trace}(EG^{-1})$.

Then $g\not= -1$ and
\begin{equation*}
(G+E)^{-1}=G^{-1}-\frac{1}{1+g} G^{-1}EG^{-1}.
\end{equation*}
\end{lemma}

%
%
%

\textbf{Inverse of symmetric matrices}
It is known that the inverse of a symmetric matrix is symmetric as well. This fact has relevance in Lemma \ref{invpart}, where 
\begin{equation*}
(A_{11}-A_{12}A_{22}^{-1}A_{21})^{-1}A_{12}A_{22}^{-1}  =  (A_{22}-A_{21}A_{11}^{-1}A_{12})^{-1}A_{21}A_{11}^{-1} ,
\end{equation*}
if $A$ is symmetric.

%
%
%

\subsection{Proofs}
\begin{proof}[Proof of Lemma \ref{l71}]
Then $U=\{ 1 \}$ and $X$ is the path matrix with reference vertex 1 of the graph. It follows that
$X_1=\text{\textbf{1}}$, $X_1' X_1=n$ and $\Pi_1=\frac{\mathbb{I}_n}{n}$, where $\mathbb{I}_n\in\R^{n\times n}$ is a matrix only consisting of ones.

We want to show that the last $s$ conlumns of 
\begin{equation*}
X_R'X_{\{1\}\cup  S_0}(X_{\{1\}\cup  S_0}'X_{\{1\}\cup  S_0})^{-1}
\end{equation*}
are the same as
\begin{equation*}
X_R'A_1 X_{  S_0}(X_{  S_0}'A_1X_{  S_0})^{-1},
\end{equation*}
i.e. the last $s$ columns of 
\begin{equation*}
X_{\{1\}\cup  S_0}(X_{\{1\}\cup  S_0}'X_{\{1\}\cup  S_0})^{-1}
\end{equation*}
are the same as
\begin{equation*}
A_1 X_{  S_0}(X_{  S_0}'A_1X_{  S_0})^{-1}.
\end{equation*}
We start by writing
\begin{equation*}
X_{\{1\}\cup  S_0}'X_{\{1\}\cup  S_0}= n \begin{pmatrix} 1 & \mu \\
\mu & \hat{\Sigma}_{ S_0 S_0} \end{pmatrix},
\end{equation*}
where $\mu$ (resp. $\mu'$) is the first column (resp. row) of $\hat{\Sigma}_{ S_0 S_0}$. Note that \begin{equation*}
X_{S_0}'A_1X_{ S_0}= n(\hat{\Sigma}_{ {S_0} {S_0}}- \mu\mu').
\end{equation*}
By using the formula for the inverse of a partitioned matrix (see Lemma \ref{invpart}) we get that
\begin{equation*}
(X_{\{1\}\cup  {S_0}}'X_{\{1\}\cup  {S_0}})^{-1}=\begin{pmatrix} \frac{1}{n(1-\mu_1)} & \frac{-1}{n(1-\mu_1)} e_1' \\
\frac{-1}{n(1-\mu_1)} e_1 & (X_ {S_0}'A_1 X_ {S_0})^{-1}
\end{pmatrix},
\end{equation*}
where $e_1=(1,0,\ldots,0)\in \R^{s}$.
As a consequence we can perform the following multiplication:
\begin{equation*}
X_{\{1\}\cup  {S_0}}(X_{\{1\}\cup  {S_0}}'X_{\{1\}\cup  {S_0}})^{-1}= \begin{pmatrix}  \frac{1}{n(1-\mu1)}(X_1-X_ {S_0} e_1) & X_ {S_0} (X_ {S_0}'A_1 X_ {S_0})^{-1}-\frac{1}{n(1-\mu_1)}X_1e_1' \end{pmatrix}.
\end{equation*}
We now develop $A_1 X_ {S_0} (X_ {S_0}'A_1X_ {S_0})^{-1}$ to see if it coincides with the second entry of the matrix we have obtained.
In particular
\begin{eqnarray*}
A_1 X_ {S_0} (X_ {S_0}' A_1 X_ {S_0})^{-1}&=& (\text{I}_n-\Pi_1)X_ {S_0}(X_ {S_0}' A_1 X_ {S_0})^{-1}\\
&=& X_ {S_0}(X_ {S_0}' A_1 X_ {S_0})^{-1}-\frac{X_1\mu' }{n}(\hat{\Sigma}-\mu\mu')^{-1}.
\end{eqnarray*}
By using Lemma \ref{lmill81} we can write the second term as
\begin{eqnarray*}
-\frac{X_1\mu' }{n}(\hat{\Sigma}_{S_0 S_0}-\mu\mu')^{-1}&=&-\frac{X_1\mu' }{n}(\hat{\Sigma}_{ {S_0} {S_0}}^{-1}+ \frac{1}{1-\mu_1})\begin{pmatrix} 1 & 0 & \ldots & 0 \\
0 & 0 & \ldots & 0 \\
\vdots & \vdots & \ddots & \vdots \\
0 & 0 & \ldots & 0
\end{pmatrix}
\\
&=& \frac{-X_1e_1'}{n}\left(1+ \frac{\mu_1}{1-\mu_1} \right)= \frac{-1}{n(1-\mu_1)}X_1 e_1'.
\end{eqnarray*}
In the KKT conditions we note that $z^0_1=0$ (indeed we have the usual normal equations for coefficients not penalized) and thus we establish the desired equality.
\end{proof}

\begin{proof}[Proof of Theorem \ref{t72}]
We refer to Section \ref{sec3} 
for the calculation of the projection coefficients.

Let us define
\begin{equation*}
\alpha(i):=\frac{i}{b^*}, i=\{1, \ldots, b^*-1 \}.
\end{equation*}
We now select an $i$ and write $\alpha=\alpha(i)$. We get that the irrepresentable condition is satisfied for a signal pattern $z\in \{-1,1\}^{K+1}$ if $\forall i \le \min\{b_1-1,b_2, \ldots, b_{K+1}\}$
\begin{enumerate}
\item $\abs{(1-\alpha,\alpha, \ldots, \alpha) z}<1$ and
\item $\abs{(\alpha,1-\alpha, -\alpha, \ldots, -\alpha) z}<1$ as well as this has to hold for any of the $K$ possible permutations of the last $K$ elements of the vector $(\alpha,1-\alpha, -\alpha, \ldots, -\alpha)'\in \R^{K+1}$.
\end{enumerate}
We now want to find the signal patterns $z$ for which the irrepresentable condition is satisfied.

Consider the first condition: it excludes the signal pattern where all the jumps have the same sign.

Thus, in the following assume w.l.o.g. that $z_1=1$.
Now we look at the second condition. We are going to consider the cases where $p$ of the $K$ last elements of the vector $(\alpha,1-\alpha, -\alpha, \ldots, -\alpha)$ get the sign $+$ and $K-p$ get the sign $-$. We look for the linear combination with the highest absolute value. This can be seen as finding the linear combination $L$ of $(\alpha, -\alpha, \ldots, -\alpha)$  determined by $p$ and then adding $\text{sgn}(L)$ to it. We scan the cases $p=1, \ldots, K-1$, since the case $p=K$ is already discarded by looking at the first condition.

For $p=1, \ldots, \lfloor (K+1)/2 \rfloor$, we have that $K+1-2p>0$, thus we assign a $+$ sign to $(1-\alpha)$ and get $1+(K+1-2p)\alpha>1$ and the irrepresentable condition is violated.

For $p=\lceil (K+1)/2 \rceil, \ldots, K-1$, we have that $K+1-2p<0$, thus we assign a $-$ sign to $(1-\alpha)$ and get $-1+(K+1-2p)\alpha<-1$ and the irrepresentable condition is violated.

If $K$ is odd, for $p=(K+1)/2$, we have that $K+1-2p=0$ and the irrepresentable condition is violated, since the linear combination gives $\pm 1$.

Thus, it only remains to consider $p=0$. For $p=0$ we get the condition $\abs{1-(K+1)\alpha}<1$ from the first as well as from the second condition above.
This condition is satisfied whenever $\alpha< 2/(K+1)$, i.e.
$$
i< \frac{2}{K+1} b^*
$$
This means that if any of $b_1-1, b_2, \ldots, b_{K+1}$ exceeds $\frac{2b^*}{K+1}$, then the irrepresentable condition is not satisfied.

\end{proof}

\newpage

\bibliographystyle{imsart-nameyear}

\bibliography{/Users/fortelli/PhD/library}

\end{document}